\documentclass[12pt,reqno]{amsart}
\usepackage{}
\usepackage{amsfonts}
\usepackage{a4wide}
\allowdisplaybreaks
\numberwithin{equation}{section}

\newtheorem{theorem}{Theorem}[section]
\newtheorem{proposition}[theorem]{Proposition}
\newtheorem{corollary}[theorem]{Corollary}
\newtheorem{lemma}[theorem]{Lemma}

\theoremstyle{definition}

\newtheorem{remark}[theorem]{Remark}

\newcommand{\R}{\mathbb{R}}
\newcommand{\N}{\mathbb{N}}

\begin{document}

\title
 [Doubly a critical system with partial weight ]
 {On a doubly critical system involving fractional Laplacian with partial weight} \footnotetext{*This work was supported by National Natural Science Foundation of China (Grant No. 11771166), Hubei Key Laboratory of Mathematical Sciences and Program for Changjiang Scholars and Innovative Research Team in University $\# $ IRT17R46, The excellent doctorial dissertation cultivation from Central China Normal University (Grant No. 2019YBZZ064).}

\maketitle
 \begin{center}
\author{Tao Yang}
\footnote{Corresponding Author: Tao Yang. T. Yang's Email addresses: yangt@mails.ccnu.edu.cn.}
\end{center}

\begin{center}
\address{Hubei Key Laboratory of Mathematical Sciences and School of Mathematics and Statistics, Central China
Normal University, Wuhan, 430079, P. R. China }
\end{center}

\maketitle
\begin{abstract}
In this paper, firstly, we consider the existence of nontrivial weak solutions to a doubly critical system involving fractional Laplacian  in $\R^{n}$ with partial weight:
\begin{equation} \label{eq0.1}
 \left\{ \begin{gathered}
  (-\Delta)^{s}u-{\gamma_1} {\frac{u}{|x'|^{2s}}}={\frac{{|u|}^{ {2^{*}_{s}}(\beta)-2}u}{|x'|^{\beta}}}+{\frac{{|u|}^{\frac{{2^{*}_{s}}(\alpha)}{2} -2}u {|v|}^{\frac{{2^{*}_{s}}(\alpha)}{2}}}{|x'|^{\alpha}}}    \hfill \\
  (-\Delta)^{s}v-{\gamma_2} {\frac{v}{|x'|^{2s}}}={\frac{{|v|}^{ {2^{*}_{s}}(\beta)-2}v}{|x'|^{\beta}}}+{\frac{{|v|}^{\frac{{2^{*}_{s}}(\alpha)}{2} -2}v {|u|}^{\frac{{2^{*}_{s}}(\alpha)}{2}}}{|x'|^{\alpha}}}     \hfill \\
\end{gathered} \right.
\end{equation}
where $s \!\in\!(0,1)$, $0\!<\! \alpha,\beta\!<\!2s\!<\!n$, $0\!< \!m\!<\!n$, $\gamma_1,\gamma_2\!<\!\gamma_{H}$, $x\!=\!(x',x'') \in \mathbb{R}^{m} \times \mathbb{R}^{n-m}$, ${2^{*}_{s}}(\alpha)\!=\!\frac{2(n-\alpha)}{n-2s}$ and $\gamma_{H}\!=\!\gamma_{H}(n,m,s)>0$ is some explicit constant.

To this end, we establish a new improved Sobolev inequality based on a weighted Morrey space. To be precise, there exists $C=C(n,m,s,\alpha)>0$ such that for any $u,v \in {\dot{H}}^s(\R^{n})$ and for any $\theta \in (\bar{\theta},1)$, it holds that
\begin{equation} \label{eq0.3}
\Big( \int_{ \R^{n} }  \frac{ |(uv)(y)|^{\frac{2^*_{s}(\alpha)}{2} } }  {  |y'|^{\alpha} } dy  \Big)^{ \frac{1}{  2^*_{s} (\alpha)  }}  \leq C ||u||_{{\dot{H}}^s(\R^{n})}^{\frac{\theta}{2}}
 ||v||_{{\dot{H}}^s(\R^{n})}^{\frac{\theta}{2}} ||(uv)||^{\frac{1-\theta}{2}}_{  L^{1,n-2s+r}(\R^{n},|y'|^{-r}) },
\end{equation}
where $s \!\in\! (0,1)$, $0\!<\!\alpha\!<\!2s\!<\!n$, $2s\!<\!m\!<\!n$, $\bar{\theta}=\max \{ \frac{2}{2^*_{s}(\alpha)}, 1-\frac{\alpha}{s}\cdot\frac{1}{2^*_{s}(\alpha)}, \frac{2^*_{s}(\alpha)-\frac{\alpha}{s}}{2^*_{s}(\alpha)-\frac{2\alpha}{m}} \}$, $r=\frac{2\alpha}{ 2^*_{s}(\alpha) }$ and $y\!=\!(y',y'') \in \mathbb{R}^{m} \times \mathbb{R}^{n-m}$.

By using mountain pass lemma and (\ref{eq0.3}), we obtain a nontrivial weak solution to (\ref{eq0.1}) in a direct way. Secondly, we extend inequality (\ref{eq0.3}) to more general forms on purpose of studying some general systems with partial weight, involving p-Laplacian especially.



\end{abstract}

\maketitle

\section{Introduction and Main Result}

\setcounter{equation}{0}
In this paper, firstly, we consider the existence of nontrivial weak solutions to a doubly critical system involving fractional Laplacian in $\R^{n}$ with partial weight:
\begin{equation} \label{eq1.1}
\left\{ \begin{gathered}
  (-\Delta)^{s}u-{\gamma_1} {\frac{u}{|x'|^{2s}}}={\frac{{|u|}^{ {2^{*}_{s}}(\beta)-2}u}{|x'|^{\beta}}}+{\frac{{|u|}^{\frac{{2^{*}_{s}}(\alpha)}{2} -2}u {|v|}^{\frac{{2^{*}_{s}}(\alpha)}{2}}}{|x'|^{\alpha}}}    \hfill \\
  (-\Delta)^{s}v-{\gamma_2} {\frac{v}{|x'|^{2s}}}={\frac{{|v|}^{ {2^{*}_{s}}(\beta)-2}v}{|x'|^{\beta}}}+{\frac{{|v|}^{\frac{{2^{*}_{s}}(\alpha)}{2} -2}v {|u|}^{\frac{{2^{*}_{s}}(\alpha)}{2}}}{|x'|^{\alpha}}}     \hfill \\
\end{gathered} \right.
\end{equation}
where $s \!\in\!(0,1)$, $0\!<\! \alpha,\beta\!<\!2s\!<\!n$, $0\!< \!m\!<\!n$, $\gamma_1,\gamma_2\!<\!\gamma_{H}$, $x\!=\!(x',x'') \in \mathbb{R}^{m} \times \mathbb{R}^{n-m}$, ${2^{*}_{s}}(\alpha)\!=\!\frac{2(n-\alpha)}{n-2s}$ and $\gamma_{H}\!=\!\gamma_{H}(n,m,s)>0$ is some explicit constant. Secondly, we will extend the existence results to some general systems with partial weight, involving p-Laplacian especially. We refer to the weight $\frac{1}{|x'|^{(\cdot)}}$ as \textbf{partial weight}.
Noticing that ${2^{*}_{s}}(\alpha)$ is the critical fractional Hardy-Sobolev exponent and  $\gamma_{H}$ is the best fractional Hardy constant on $\R^{n}$ (See Lemmas \ref{lemma2.1}-\ref{lemma2.2}). The fractional Laplacian $(-\Delta)^{s}$ is defined on the Schwartz class (space of rapidly decaying $C^{\infty}$ functions in $\R^{n}$) through Fourier transform,
$$\widehat{(-\Delta)^{s}u}(\xi)= |\xi|^{2s} \hat{u}(\xi), \forall \xi \in \R^{n},      $$
where $\hat{u}(\xi)=\frac{1}{(2\pi)^{n/2}} \int_{\R^{n}} e^{-i\xi x}u(x)dx$ is the Fourier transform of $u$.

Throughout this paper, for $0\!\leq\!m\!\leq\!n$ and $y\!=\!(y',y'') \in \mathbb{R}^{m} \times \mathbb{R}^{n-m}$, we denote the norm of
$L^{p}(\R^{n},|y'|^{-\lambda})$ by $${||u||}_{L^{p}(\R^{n},|y'|^{-\lambda})}:=\Big ( \int_{\R^{n}}  \frac{|u(y)|^p}{  |y'|^{\lambda} }dy  \Big )^{\frac{1}{p}}$$
for any $0 \leq \lambda<n$ and $p\in[1,+\infty)$. 
The space ${\dot{H}}^s(\R^{n})$ is defined as the completion of $C_{0}^{\infty}(\R^{n})$ under the norm
$$   ||u||_{\dot{H}^{s}(\R^{n})}= \Big(\int_{\R^{n}}|\xi|^{2s}|\hat{u}(\xi)|^2d\xi\Big )^{\frac{1}{2}}= \Big(\int_{\R^{n}} |(-\Delta)^{s/2}u|^2dx\Big )^{\frac{1}{2}}.$$
The dual space of ${\dot{H}}^s(\R^{n})$ is denoted by ${{\dot{H}}^s(\R^{n})^{'}}$. See \cite{ENEV} and references therein for the basics on the fractional Laplacian.

The fractional and non-local operators arise in many different applications,
see \cite{ENEV,Caff} and the references therein. For \eqref{eq1.1} with $s=1$, $m=n$ and $\alpha=0=\beta$, it is related to Bose-Einstein condensate, see \cite{CzwZ}. In \cite{mBgT,GmKs}, a class of partial weight problems were studied as a model describing the dynamics of galaxies.

For any $(u,v) \in X=\dot{H}^{s}(\R^{n})\times\dot{H}^{s}(\R^{n})$, the energy functional of (\ref{eq1.1}) is defined as:
\begin{align*}
  I(u,v)&=\frac{1}{2}\int_{\R^{n}} \big[{|(-\Delta)^{\frac{s}{2}}u|}^2
-{\gamma_1} {\frac{u^2}{|x'|^{2s}}}+{|(-\Delta)^{\frac{s}{2}}v|}^2
-{\gamma_2} {\frac{v^2}{|x'|^{2s}}}\big]dx  \\
&-\frac{2}{2^{*}_{s}(\alpha)}\int_{\R^n}{\frac{{|uv|}^{ \frac{{2^{*}_{s}}(\alpha)}{2} }}{|x'|^{\alpha}}}dx
-\frac{1}{ 2^{*}_{s}(\beta) }\int_{\R^{n}}{\frac{{|u|}^{ {2^{*}_{s}}(\beta)}+{|v|}^{ {2^{*}_{s}}(\beta)}}{|x'|^{\beta}}}dx.
\end{align*}
A nontrivial critical point of $I$ is a nontrivial weak solution to (\ref{eq1.1}). We say a pair of functions $(u,v) \in X$ is a \textbf{nontrivial solution} of \eqref{eq1.1} if
$$u\not\equiv0,~~~~v\not\equiv0,~~~~{\langle I'(u,v),(\phi,\psi) \rangle}_X=0,~~~~\forall(\phi,\psi) \in X;$$
if $(u,v)= (u, 0)$ or $(u, v) = (0, v)$ or $u\equiv v$, we say that $(u,v) \in X$ is a \textbf{semi-nontrivial solution} of \eqref{eq1.1}, where ${\langle ,\rangle}_X$ denote the inner product in $X$(See formula \eqref{DeF2.4}).

The problem of multiple critical exponents has been extensively studied by scholars, see \cite{RFPP,NGSS,JYFW,NGCY,DAEJ,DKGL,CWJ,HYKD,JWJP,CzwZ,gBtY,GbTs}. Dating back to \cite{RFPP}, R. Filippucci et al. studied the doubly critical equation of Emden-Fowler type:
\begin{equation} \label{eq1.2}
 -{\Delta}_pu-{\kappa} {\frac{u^{p-1}}{|x|^{p}}}=u^{p^{*}-1}+\frac{u^{ {p^{*}(\alpha)-1}}}{|x|^{\alpha}} ~~~~\mbox{in}~~{\R}^n, u \geq 0, \ \ u \in {D}^{1,p}(\R^{n})
 \end{equation}
where $n\geq 2$, $p \in (1,n)$, $\alpha \in (0,p)$, $0\leq \kappa < \bar{\kappa}=(\frac{n-p}{p})^p$, $p^{*}=\frac{np}{n-p}$, $p^{*}(\alpha)=\frac{p(n-\alpha)}{n-p}$ and ${\Delta}_p u:=div({|\nabla u|}^{p-2}\nabla u)$ is the p-Laplacian of $u$. The space ${D}^{1,p}(\R^{n})$ is defined as the completion of $C^{\infty}_{0}({\R}^n)$ under the norm $$||u||_{{D}^{1,p}(\R^{n})}=(\int_{\R^{n}} |\nabla u|^p dx)^{\frac{1}{p}}.$$
Based on truncation skills, they obtain a nontrivial weak solution to $(\ref{eq1.2})$ by using mountain pass lemma and concentration analysis of the corresponding $(PS)$ sequence.
The works of \cite{NGSS,JYFW,gBtY} were devoted to the fractional Laplacian equations involving different critical nonlinearities, one can refer to \cite{LwBz,GbTs} for fractional Laplacian systems involving different critical nonlinearities. For the cases of the standard Laplacian, biharmonic operator and p-biharmonic operator, the interested reader can refer to \cite{CzwZ,FCZQ,JLCS,NGAM, NGFR,NGCY,YWYS,DAEJ,AEKS}.

Let us focus on partial weight problems. For $n\geq 3$, $\alpha \in (0,n)$, $2\leq m\leq n$, $x\!=\!(x',x'') \in \mathbb{R}^{m} \times \mathbb{R}^{n-m}$ and  $2^{*}(\alpha)=\frac{2(n-\alpha)}{n-2}$, the minimization problem
\begin{equation} \label{galaxies1.3}
\bar{\Lambda}(n,m,\alpha)=  \mathop {\inf }\limits_{u \in {D}^{1,2}(\R^{n}) \setminus \{0\}}   \frac{\int_{\R^{n}} |\nabla u|^2 dx}{\Big( \int_{\R^{n}} \frac{{|u|}^{ { 2^{*} }(\alpha)}}{|x'|^{\alpha}}dx  \Big)^{\frac{2}{  2^{*}(\alpha)  }}},
\end{equation}
is attained by a cylindrically symmetric decreasing function, see \cite{mBgT,GmKs} for details. A minimizer of $\bar{\Lambda}(n,m,\alpha)$ weakly solves the problem
\begin{equation} \label{Galaxies1.4}
-\Delta u={\frac{{|u|}^{ {2^{*}}(\alpha)-2}u}{|x'|^{\alpha}}}
\end{equation}
up to a multiplying constant. Equation \eqref{Galaxies1.4} is related to a model describing the dynamics of elliptic galaxies, see \cite{BgER,lcTi}. Existence and symmetry of solutions were studied in \cite{mBgT,GmKs}. In the case $\alpha=1$, all positive finite energy solutions of \eqref{Galaxies1.4} were classified in \cite{IFaK}. Such solutions are given by $V(x',x'')=\frac{{\big( (n-2)(m-1)\big)}^{\frac{n-2}{2}}}{{\big( (1+{|x'|})^2+{|x''|}^2\big)}^{\frac{n-2}{2}}}$ or its scaling and translations in the $x''$-variable. In \cite{XcTy}, X. L. Chen et al. extended \eqref{galaxies1.3} to nonlocal case and considered
\begin{equation} \label{sgalaxies1.5}
\bar{\Lambda}(n,m,s,\alpha)=  \mathop {\inf }\limits_{u \in \dot{H}^{s}(\R^{n}) \setminus \{0\}  }  \frac{\int_{\R^{n}} |(-\Delta)^{s/2}u|^2dx}{\Big( \int_{\R^{n}} \frac{{|u|}^{ { 2_s^{*} }(\alpha)}}{|x'|^{\alpha}}dx  \Big)^{\frac{2}{  2_s^{*}(\alpha)  }}},
\end{equation}
where $s \!\in\!(0,1)$, $0\!<\!\alpha\!<\!2s\!<\!n$, $0\!< \!m\!<\!n$, $x\!=\!(x',x'') \in \mathbb{R}^{m} \times \mathbb{R}^{n-m}$ and ${2^{*}_{s}}(\alpha)\!=\!\frac{2(n-\alpha)}{n-2s}$. They showed that \eqref{sgalaxies1.5} is achieved by a positive, cylindrically symmetric and strictly decreasing function $u(x)$, which weakly solves the problem
\begin{equation} \label{sGalaxies1.6}
(-\Delta)^s u={\frac{{|u|}^{ {2_s^{*}}(\alpha)-2}u}{|x'|^{\alpha}}}
\end{equation}
up to a multiplying constant. Decaying laws for the minimizer $u$ were also established.

Motivated by the above papers, we study the existence of nontrivial weak solutions to system $(\ref{eq1.1})$ with partial weight. To our best knowledge, $(\ref{eq1.1})$ has not been studied before.

Our first main results are as follows (Our second main results will be stated and proved in Section 6):

\begin{theorem} \label{th1.1}
Let $\gamma_1\not=\gamma_2$, then system (\ref{eq1.1}) possesses at least a nontrivial weak solution provided either \textbf{(I)} $s \in(0,1)$, $0<\alpha,\beta<2s<n$, $2s<m<n$ and $\gamma_1,\gamma_2<\gamma_{H}$ or \textbf{(II)} $s \in(0,1)$, $\beta=0<\alpha<2s<n$, $2s<m<n$ and $0\leq \gamma_1,\gamma_2<\gamma_{H}$.
\end{theorem}

\begin{remark}
Theorem \ref{th1.1} indicates that we can relax the lower bound of $\gamma_1,\gamma_2$ in system (\ref{eq1.1}) provided $\alpha,\beta>0$ since we need the extra condition $\gamma_1,\gamma_2\geq0$ if $\beta=0$. 
\end{remark}

There are four main difficulties in the proof of Theorem \ref{th1.1}. \textbf{Firstly}, truncation skills used in \cite{RFPP,NGSS} do not work if we choose $X\!=\!\dot{H}^{s}(\R^{n})\times \dot{H}^{s}(\R^{n})$ as the work space since $(-\Delta)^{s}$ is a nonlocal operator. Although the $s$-harmonic extension (obtained by L. Caffarelli et al. in \cite{LCLS}) can overcome the difficulty of the non-locality of $(-\Delta)^{s}$,  
this method is more complicated and less straightforward since the appearance of the partial weight terms $\frac{1}{|x'|^{(\cdot)}}$ in (\ref{eq1.1}). \textbf{Secondly}, the compactness of the corresponding $(PS)$ sequence can not hold for
any energy level $c>0$ since system (\ref{eq1.1}) is invariant under the transformation
$$\big(u(x), v(x)\big)\mapsto \big({\lambda}^{ \frac{n-2s}{2} }u({\lambda}x',{\lambda}(x''-y'') ),{\lambda}^{ \frac{n-2s}{2} }v({\lambda}x',{\lambda}(x''-y'') )\big), ~~~~~~~~\forall x\!=\!(x',x'') \in \mathbb{R}^{m} \times \mathbb{R}^{n-m},$$
where ${\lambda}>0$ and $y''\in \mathbb{R}^{n-m}$. In fact, assume by contradiction that the compactness of the corresponding $(PS)$ sequence holds for some $c > 0$, and let $\{(u_k, v_k)\}$ be a $(PS)_c$ sequence, that is, $I(u_k, v_k) \to c$ and $I'(u_k, v_k) \to 0$ as $k \to +\infty$. Then up to a subsequence, we may assume that $(u_k, v_k) \to (u, v)$ strongly in $X$. Define $\big(\tilde{u}_k(x), \tilde{v}_k(x)\big)=\big(k^{ \frac{n-2s}{2} }u_k(k x),k^{ \frac{n-2s}{2} }v_k(k x)\big)$ and $\big(\bar{u}_k(x), \bar{v}_k(x)\big)=\big(\tilde{u}_k(x',x''-y''_k), \tilde{v}_k(x',x''-y''_k)\big)$ for $y''_k\in \mathbb{R}^{n-m}$ satisfying $|y''_k|\to+\infty$ as $k\to+\infty$, then it is easy to check that $\{(\bar{u}_k, \bar{v}_k)\}$ is also a $(PS)c$
sequence and $(\bar{u}_k, \bar{v}_k) \rightharpoonup (0, 0)$ weakly in $X$. But then, we have $(\bar{u}_k, \bar{v}_k) \rightarrow(0, 0)$ strongly in $X$, which contradicts with $c>0$. \textbf{Thirdly}, there is an asymptotic competition between the energy carried by the two critical nonlinearities, so we have trouble in ruling out the "vanishing" of the corresponding $(PS)$ sequence. \textbf{Fourthly}, the appearance of the coupled terms ${\frac{{|u|}^{\frac{{2^{*}_{s}}(\alpha)}{2} -2}u {|v|}^{\frac{{2^{*}_{s}}(\alpha)}{2}}}{|x'|^{\alpha}}}$ and ${\frac{{|v|}^{\frac{{2^{*}_{s}}(\alpha)}{2} -2}v {|u|}^{\frac{{2^{*}_{s}}(\alpha)}{2}}}{|x'|^{\alpha}}} $ in (\ref{eq1.1}) brings in more difficulty for searching nontrivial solutions.

The method adopted in \cite{JYFW,LwBz,gBtY,GbTs} is not applicable to (\ref{eq1.1}) since the existence of the partial weight terms $\frac{1}{|x'|^{(\cdot)}}$ in (\ref{eq1.1}). For these reasons, we develop new tools which is based on a weighted Morrey space, see Section 3. To be precise, we discover the embeddings
\begin{equation}  \label{eq1.06}
{\dot{H}}^s(\R^{n}) \hookrightarrow  {L}^{2^*_{s}(\alpha)}(\R^{n},|y'|^{-\alpha}) \hookrightarrow L^{2,n-2s+r}(\R^{n},|y'|^{-r}),
\end{equation}
where $s \in (0,1)$, $0<\alpha<2s<n$, $0<m<n$, $y\!=\!(y',y'') \in \mathbb{R}^{m} \times \mathbb{R}^{n-m}$ and $r=\frac{2\alpha}{ 2^*_{s}(\alpha) }$; Based on \eqref{eq1.06}, we establish the following improved Sobolev inequalities with partial weight:


\begin{proposition}\label{prop1.4}
Let $s \!\in \!(0,1)$, $0\!<\!\alpha\!<\!2s\!<\!n$  and $2s\!<\!m\!<\!n$. There exists $C\!=\!C(n,m,s,\alpha)\!>\!0$ such that for any $u,v \!\in\! {\dot{H}}^s(\R^{n})$ and for any $\theta \!\in\! (\bar{\theta},1)$, it holds that
\begin{equation}  \label{eqa1.6}
 \Big( \int_{ \R^{n} }  \frac{ |(uv)(y)|^{\frac{2^*_{s}(\alpha)}{2} } }  {  |y'|^{\alpha} } dy  \Big)^{ \frac{1}{  2^*_{s} (\alpha)  }}  \leq C ||u||_{{\dot{H}}^s(\R^{n})}^{\frac{\theta}{2}}
 ||v||_{{\dot{H}}^s(\R^{n})}^{\frac{\theta}{2}} ||(uv)||^{\frac{1-\theta}{2}}_{  L^{1,n-2s+r}(\R^{n},|y'|^{-r}) },
\end{equation}
where $y\!=\!(y',y'') \in \mathbb{R}^{m} \times \mathbb{R}^{n-m}$, $\bar{\theta}=\max \{ \frac{2}{2^*_{s}(\alpha)}, 1-\frac{\alpha}{s}\cdot\frac{1}{2^*_{s}(\alpha)}, \frac{2^*_{s}(\alpha)-\frac{\alpha}{s}}{2^*_{s}(\alpha)-\frac{2\alpha}{m}} \}$ and $r=\frac{2\alpha}{ 2^*_{s}(\alpha) }$.
\end{proposition}


\begin{corollary}\label{coro1.5}
Let $n\!\geq\!2$, $2\!\leq\! p\!<\!n$ and $0\!<\!\alpha\!<\!p\!<\!m\!<\!n$. There exists $C\!=\!C(n,m,p,\alpha)\!>\!0$ such that for any $u,v \!\in\! {D}^{1,p}(\R^{n})$ and for any $\theta \!\in\! (\bar{\theta},1)$, it holds that
\begin{equation}  \label{eq1.7}
 \Big( \int_{ \R^{n} }  \frac{  {|(uv)|}^  {\frac{ p^*(\alpha)}{2} }   }   {  |y'|^{\alpha}  } dy  \Big)^{ \frac{1}{  p^* (\alpha)  }}  \leq C ||u||_{{D}^{1,p}(\R^{n})}^{\frac{\theta}{2}}
 ||v||_{{D}^{1,p}(\R^{n})}^{\frac{\theta}{2}} ||(uv)||^{\frac{1-\theta}{2}}_{  L^{{p}/{2},n-p+r}(\R^{n},|y'|^{-r}) },
\end{equation}
where $y\!=\!(y',y'') \in \mathbb{R}^{m} \times \mathbb{R}^{n-m}$, $\bar{\theta}=\max \{ \frac{p}{p^*(\alpha)}, 1-\frac{\alpha}{p^*(\alpha)},\frac{p^*(\alpha)
-\alpha}{p^*(\alpha)-\frac{p\alpha}{m}} \}$ and $p^{*}(\alpha)=\frac{p(n-\alpha)}{n-p}$.
\end{corollary}
\begin{remark}
Proposition \ref{prop1.4} and Corollary \ref{coro1.5} are more general than Theorems 1-2 by G. Palatucci et al. in \cite{GPAP}; The detailed proof will be given in Section 3.
\end{remark}

Now, we give the outline of the proof for Theorem \ref{eq1.1}. We use the Mountain pass lemma to find critical points of $I(u,v)$ on $X=\dot{H}^{s}(\R^{n})\times\dot{H}^{s}(\R^{n})$, which correspond to weak solutions for system (\ref{eq1.1}). Since problem (\ref{eq1.1}) includes double critical exponents, we require the Mountain pass level $c<c^*$ for some suitable threshold value $c^*$. This is crucial in  ruling out the "vanishing" of the corresponding (PS) sequence. To this end, we introduce the minimization problems
\begin{equation}  \label{eq1.9}
    {S}(n,s,\alpha)=\mathop {\inf }\limits_{(u,v) \in X \setminus \{(0,0)\}  }    \frac{ \int_{\R^{n}} \big[{|(-\Delta)^{\frac{s}{2}}u|}^2
-{\gamma_1} {\frac{u^2}{|y'|^{2s}}}+{|(-\Delta)^{\frac{s}{2}}v|}^2
-{\gamma_2} {\frac{v^2}{|y'|^{2s}}}\big]dy }
   { \Big( \int_{\R^{n}} \frac{{|u|}^{ { 2^{*}_{s} }(\alpha)}}{|y'|^{\alpha}}dy+\int_{\R^{n}} \frac{{|u|}^{ { 2^{*}_{s} }(\alpha)}}{|y'|^{\alpha}}dy   \Big)^{\frac{2}{ 2_s^*(\alpha)}} }
\end{equation}
and
\begin{equation}  \label{eq1.8}
  {\Lambda}(n,s,\alpha)=  \mathop {\inf }\limits_{(u,v) \in X \setminus \{(0,0)\}  }   \frac{ \int_{\R^{n}} \big[{|(-\Delta)^{\frac{s}{2}}u|}^2
-{\gamma_1} {\frac{u^2}{|y'|^{2s}}}+{|(-\Delta)^{\frac{s}{2}}v|}^2
-{\gamma_2} {\frac{v^2}{|y'|^{2s}}}\big]dy }{\Big( \int_{\R^{n}} \frac{{|uv|}^{   \frac{2^{*}_{s}(\alpha)}{2}     }}{|y'|^{\alpha}}dy\Big)^{\frac{2}{  2^{*}_{s}(\alpha)  }}}.
\end{equation}
Using the minimizers of $S(n,s,\beta)$ and $\Lambda(n,s,\alpha)$, we can prove the Mountain pass level $c<c^*:=\min \Big \{ \frac{2s-\beta}{2(n-\beta)} {S(n,s,\beta)}^{\frac{n-\beta}{2s-\beta}}, \frac{2s-\alpha}{n-\alpha} \Big[\frac{\Lambda(n,s,\alpha)}{2}\Big]^{\frac{n-\alpha}{2s-\alpha}} \Big \}$. Then, the Mountain pass lemma gives a $(PS)_c$ sequence $\{(u_k,v_k)\}_{k=1}^{+\infty}$ for $I$ at level $c>0$, i.e.
\begin{equation} \label{eq1.013}
\lim_{k \to +\infty}I(u_k,v_k)=c<c^*~~\mbox{and}~~ \lim_{k \to +\infty} I'(u_k,v_k)=0~~\mbox{strongly in}~~X'.
\end{equation}
Clearly, $\{(u_k,v_k)\}$ is bounded so we may assume $(u_k,v_k) \rightharpoonup (u,v)$ in $X$ for some $(u,v)\in X$. But it may occur that $u\equiv0$ or $v\equiv0$. Denote
$$d_1=\lim_{k \to +\infty}\int_{\R^n}{\frac{{|u_kv_k|}^{ \frac{{2^{*}_{s}}(\alpha)}{2} }}{|x'|^{\alpha}}}dx,~~~~d_2=\lim_{k \to +\infty}\int_{\R^{n}}    {\frac{{|u_k|}^{ {2^{*}_{s}}(\beta)}\!+\!{|v_k|}^{ {2^{*}_{s}}(\beta)}}{|x'|^{\beta}}} dx.$$
From \eqref{eq1.9}-\eqref{eq1.013}, we have
\begin{equation} \label{eq1.014}
  d_1^{\frac{2}{  2^{*}_{s}(\alpha) }}A_1\leq d_2,~~~~d_2^{\frac{2}{  2^{*}_{s}(\beta) }}A_2\leq 2d_1,
\end{equation}
where $A_1\!=\!\Lambda(n,s,\alpha)\!-\!2[\frac{n-\alpha}{2s-\alpha} c]^{\frac{2s-\alpha}{n-\alpha}}$ and $A_2\!=\!S(n,s,\beta)\!-\!  [\frac{2(n-\beta)}{2s-\beta} c]^{\frac{2s-\beta}{n-\beta}}$. Since $0<c<c^*$, we derive that $A_1>0$ and $A_2>0$. Thus $(\ref{eq1.014})$ implies that $d_1\geq{\varepsilon}_0>0$ and $d_2\geq{\varepsilon}_0>0$(If $d_1=0$ and $d_2=0$, then $c=0$, a contradiction), i.e.
$$\lim_{k \to +\infty}\int_{\R^n}{\frac{{|u_kv_k|}^{ \frac{{2^{*}_{s}}(\alpha)}{2} }}{|x'|^{\alpha}}}dx\geq {\varepsilon}_0>0,~~~~~~~~\lim_{k \to +\infty}\int_{\R^{n}}    {\frac{{|u_k|}^{ {2^{*}_{s}}(\beta)}\!+\!{|v_k|}^{ {2^{*}_{s}}(\beta)}}{|x'|^{\beta}}} dx\geq{\varepsilon}_0>0. $$
Then, the embeddings $(\ref{eq1.06})$ and the improved Sobolev inequality (\ref{eqa1.6}) imply that
$$   0<C \leq ||(u_kv_k)||_{  L^{1,n-2s+r}(\R^{n},|y'|^{-r}) }\leq C^{-1}~~~~~~ \mbox{for any}~~~~ k\geq K~~~~ \mbox{large}, $$
where $r=\frac{2\alpha}{ 2^*_{s}(\alpha) }$ and $C>0$ is a constant.
For any $k\!>\! K$, we may find ${\lambda}_k\!>\!0$ and $x_k\!=\!(x'_k,x''_k) \!\in\! \R^{m}\!\times\! \R^{n-m}$ such that
$$  {\lambda}_k^{-2s+r} \int_{B_{{\lambda}_k}(x_k)} \frac{|(u_kv_k)(y)|}{  |y'|^{r} }dy > ||(u_kv_k)||_{  L^{1,n-2s+r}(\R^{n},|y'|^{-r}) } -\frac{C}{2k} \geq C_1>0,$$
where $y=(y',y'')\in \R^{m}\times \R^{n-m}$. Let $\tilde{u}_k(x)={\lambda}_k^{ \frac{n-2s}{2} }u_k\big({\lambda}_kx', {\lambda}_k(x''-x''_k)\big)$, $\tilde{v}_k(x)={\lambda}_k^{ \frac{n-2s}{2} }v_k\big({\lambda}_kx', {\lambda}_k(x''-x''_k)\big)$ and ${\tilde{x}}_k=\frac{x_k}{{\lambda}_k}
=({\tilde{x}}'_k,{\tilde{x}}''_k)\in \R^{m}\times \R^{n-m}$, then we have
 $$(\tilde{u}_k,\tilde{v}_k) \rightharpoonup (\tilde{u},\tilde{v}) ~~~~\mbox{in}~~~~X, ~~~~~\tilde{u}\not \equiv 0,~~~~~~\tilde{v}\not \equiv 0.$$
In fact, we can prove that $\{{\tilde{x}}'_k\}$ is bounded and so there exists $R>0$ such that
\begin{equation}\label{eq1.015}
\int_{B_{R}(0)} \frac{|(\tilde{u}_k \tilde{v}_k)(x)|}{  |x'|^{r} }dx \geq \int_{B_{1}(({\tilde{x}}'_k,\vec{0}))} \frac{|(\tilde{u}_k \tilde{v}_k)(x)|}{  |x'|^{r} }dx  \geq C_1>0,
\end{equation}
where $x=(x',x'')\in \R^{m}\times \R^{n-m}$ and $\vec{0}=(0,\cdots,0)\in \R^{n-m}$.
From (\ref{eq1.015}), we have $\int_{B_{R}(0)} \frac{|(\tilde{u} \tilde{v})(x)|}{  |x'|^{r} }dx  \geq C_1>0$ since $\frac{r}{2}<s$.
Moreover, we can check that
$$\lim_{k \to +\infty}I(\tilde{u}_k,\tilde{v}_k)=c~~~~\mbox{and}~~~~ I'(\tilde{u},\tilde{v})=\lim_{k \to +\infty} I'(\tilde{u}_k,\tilde{v}_k)=0~~\mbox{strongly in}~~X'.$$

It remains to deal with the minimization problems (\ref{eq1.9})-(\ref{eq1.8}). To this end, we need some kind of compactness. Problem (\ref{eq1.9}) can be solved by using the embeddings (\ref{eq1.06}) and the inequality \eqref{eqa1.6} with $u=v$, see the proof of Proposition \ref{pro1.7}-(2)(3). Problem (\ref{eq1.8}) is more difficult since the appearance of the coupled term $\int_{\R^{n}} \frac{{|uv|}^{   \frac{2^{*}_{s}(\alpha)}{2}     }}{|y'|^{\alpha}}dy$ in the right hand side of  ${\Lambda}(n,s,\alpha)$, see the proof of Proposition \ref{pro1.7}-(1). \textbf{Thanks to the embeddings \eqref{eq1.06} and the new inequality \eqref{eqa1.6}, we can prove the existence of minimizers for ${S}(n,s,\alpha)$ and ${\Lambda}(n,s,\alpha)$ in $X$ in a direct way. Moreover, (\ref{eq1.06}) and (\ref{eqa1.6}) are useful to rule out the "vanishing" of the corresponding $(PS)$ sequence.} Finally, we extend inequality (\ref{eqa1.6}) to more general forms to study some general systems with partial weight, involving p-Laplacian especially. As far as we know, the strategy we adopt is new.

The rest of the paper is organized as follows: in Section 2, we give some preliminaries. In Section 3, we introduce the weighted Morrey space and establish improved Sobolev inequalities, i.e., we prove Proposition  \ref{prop1.4} and Corollary \ref{coro1.5}. In Section 4, we solve the minimization problems (\ref{eq1.9})-(\ref{eq1.8}). In Section 5, we prove Theorem \ref{th1.1}. In Section 6, we extend the main results of \eqref{eq1.1} to some general systems with partial weight. \\

\textbf{Notation:}  We use $\rightarrow$ and $\rightharpoonup$ to denote the strong and weak convergence in the corresponding spaces respectively. Write "Palais-Smale" sequences as $(PS)$ sequences in short. $\mathbb{N}=\{1,2,\cdots\}$ is the set of natural numbers. $\mathbb{R}$ and $\mathbb{C}$ denote the sets of real and complex numbers respectively. By saying a function is "measurable", we always mean that the function is "Lebesgue" measurable. "$\wedge$" denotes the Fourier transform and "$\vee$" denotes the inverse Fourier transform. Generic fixed and numerical constants will be denoted by $C$(with subscript in some case) and they will be allowed to vary within a single line or formula.

\section{Preliminaries}
\setcounter{equation}{0}
In this section, we give some preliminary results.

\begin{lemma}\label{lemma2.1} (Fractional Hardy inequality with partial weight: Formula (3.2) in \cite{XcTy}) \\
Let $s\in (0,1)$, $n>2s$ and $0< m<n$. Then there exists $\gamma_{H}=\!\gamma_{H}(n,m,s)>0$ such that
\begin{equation}\label{eq2.1}
  {\gamma}_{H} \int_{\R^{n}} {\frac{u^2}{|x'|^{2s}}}dx \leq \int_{\R^{n}} |(-\Delta)^{s/2}u|^2dx, ~~~~ \forall u \in {\dot{H}}^s(\R^{n})
\end{equation}
where $x\!=\!(x',x'') \in \mathbb{R}^{m} \times \mathbb{R}^{n-m}$,  $\gamma_{H}\!=\!\frac{2 \pi^{\frac{n-m}{2}} \Gamma\left(\frac{m+2 s}{2}\right)}{\Gamma\left(\frac{n+2 s}{2}\right)} \int_{0}^{1} r^{2 s-1}\left|1\!-\!r^{\frac{m-2 s}{2}}\right|^{2} \Phi_{m,s}(r) d r\!>\!0  $,
$\Phi_{m, s}(r):=\!\left|\mathbb{S}^{m-2}\right| \int_{-1}^{1} \frac{\left(1-t^{2}\right)^{\frac{m-3}{2}}}{\left(1-2 r t+r^{2}\right)^{\frac{m+2 s}{2}}} d t$ for $m \!\geq\! 2$ and  $\Phi_{m,s}(r):=\!\frac{1}{(1-r)^{1+2 s}}+\frac{1}{(1+r)^{1+2 s}}$ for $ m\!=\!1$. Here $\Gamma$ denotes the Gamma function and $\mathbb{S}^{m-2}$ denotes the surface of a $(m-1)$ dimensional unit ball.
\end{lemma}

\begin{lemma}\label{lemma2.2} (Fractional Hardy-Sobolev inequality with partial weight: Lemma 3.1 of \cite{XcTy}) \\
Let $s\in (0,1)$, $0< \alpha \leq 2s<n$ and $0< m<n$. Then there exist positive constants $c$ and $C$ such that
\begin{equation}\label{eq2.2}
   \Big( \int_{\R^{n}} \frac{{|u|}^{ { 2^{*}_{s} }(\alpha)}}{|x'|^{\alpha}}dx  \Big)^{\frac{2}{  2^{*}_{s}(\alpha)  }} \leq c \int_{\R^{n}} |(-\Delta)^{s/2}u|^2dx,~~~~\forall u \in {\dot{H}}^s(\R^{n}).
\end{equation}
Moreover, if $\gamma<\gamma_{H}=\gamma_{H}(n,m,s)$, then
\begin{equation}\label{eq2.3}
  C \Big( \int_{\R^{n}} \frac{{|u|}^{ { 2^{*}_{s} }(\alpha)}}{|x'|^{\alpha}}dx  \Big)^{\frac{2}{  2^{*}_{s}(\alpha)  }} \leq \int_{\R^{n}} |(-\Delta)^{s/2}u|^2dx-{\gamma} \int_{\R^{n}} {\frac{u^2}{|x'|^{2s}}}dx,~~~~\forall u \in {\dot{H}}^s(\R^{n}).
\end{equation}
Here $x\!=\!(x',x'') \in \mathbb{R}^{m} \times \mathbb{R}^{n-m}$.\\
\end{lemma}

From Lemma \ref{lemma2.1}, the following inequality holds for all $\gamma_i<\gamma_{H}$ and any $u \in {\dot{H}}^s(\R^{n})$,
\begin{equation}\label{eq2.4}
  (1-\frac{\gamma_{i}^{+}}{\gamma_{H}})\int_{\R^{n}} |(-\Delta)^{s/2}u|^2dx \leq ||u||_{\gamma_i}^2 \leq (1+\frac{\gamma_{i}^{-}}{\gamma_{H}}) \int_{\R^{n}} |(-\Delta)^{s/2}u|^2dx,~~~~i=1,2
\end{equation}
where $||u||_{\gamma_i}\!=\!{\Big(\int_{\R^{n}} |(-\Delta)^{s/2}u|^2dx\!-\!{\gamma_i} \int_{\R^{n}} {\frac{u^2}{|x'|^{2s}}}dx \Big)}^{\frac{1}{2}}$ and $\gamma_{i}^{\pm}\!=\!\max\{\pm\gamma_i,0\}$. Define two \textbf{equivalent norms} on ${\dot{H}}^s(\R^{n})$ by $||\cdot||_{\gamma_i}$ and the inner products of $u,v \in {\dot{H}}^s(\R^{n})$ by
$$ {\langle u, v \rangle}_{\gamma_i}=\int_{\R^{n}} (-\Delta)^{\frac{s}{2}}u(-\Delta)^{\frac{s}{2}}vdx-{\gamma_i} \int_{\R^{n}} {\frac{uv}{|x'|^{2s}}}dx,~~~~i=1,2.$$
For any $(u,v),(\phi,\psi) \in X=\dot{H}^{s}(\R^{n})\times\dot{H}^{s}(\R^{n})$, we define
\begin{equation}\label{DeF2.4}
{\langle (u,v), (\phi,\psi) \rangle}_X={\langle u, \phi \rangle}_{\gamma_1}+{\langle v, \psi \rangle}_{\gamma_2},~~~~||(u,v)||^2=||u||_{\gamma_1}^2+||v||_{\gamma_2}^2.
\end{equation}\\

\begin{lemma}  \label{lemma2.3}
 Let $s \in (0,1)$, $0<r<s<\frac{n}{2}$ and $0< m<n$. If $\{u_k\}$ is a bounded sequence in ${\dot{H}}^s(\R^{n})$ and $u_k \rightharpoonup   u ~~\mbox{in}~~{\dot{H}}^s(\R^{n})$, then
$$\frac{|u_k|}{ {|x'|}^{r}  }    \rightarrow   \frac{|u|}{ {|x'|}^{r}  } ~~\mbox{in}~~L^2_{loc}(\R^{n}),~~~~~~~~\frac{u_k}{ {|x'|}^{r}  }    \rightarrow   \frac{u}{ {|x'|}^{r}  } ~~\mbox{in}~~L^2_{loc}(\R^{n}),$$
where $x\!=\!(x',x'') \in \mathbb{R}^{m} \times \mathbb{R}^{n-m}$.
\end{lemma}

\begin{proof}
It is similar to the proof of Lemma 2.3 in \cite{gBtY}.
\end{proof}

\begin{lemma}\label{LmMa2.5}
Let $s \in(0,1)$, $0<\alpha<2s<n$ and $0< m<n$. If $\{(u_k,v_k)\}_{k \in \mathbb{N}}$ is a bounded sequence in $X=\dot{H}^{s}(\R^{n})\times\dot{H}^{s}(\R^{n})$ and $(u_k,v_k) \rightharpoonup (u,v)$ in $X$, then we have
\begin{align*}
\mathop {\lim }\limits_{k  \to \infty} \int_{\R^{n}} \frac{{|{u}_k{v}_k|}^{   \frac{2^{*}_{s}(\alpha)}{2}     }}{|x'|^{\alpha}}dx=\int_{\R^{n}} \frac{{|{u}{v}|}^{   \frac{2^{*}_{s}(\alpha)}{2}     }}{|x'|^{\alpha}}dx+\mathop {\lim }\limits_{k  \to \infty}\int_{\R^{n}} \frac{{|({u}_k-{u})({v}_k-{v})|}^{   \frac{2^{*}_{s}(\alpha)}{2}     }}{|x'|^{\alpha}}dx,
\end{align*}
where $x\!=\!(x',x'') \in \mathbb{R}^{m} \times \mathbb{R}^{n-m}$.
\end{lemma}

\begin{proof}
The proof is similar to that of Theorem 2 in \cite{HBZI} when using $j:\mathbb{C} \to \mathbb{C}$ defined by $j(s + it):=\frac{{|{s}{t}|}^{   \frac{2^{*}_{s}(\alpha)}{2}     }}{|x'|^{\alpha}}$, for $s, t \in \mathbb{R}$, in that theorem. Here $i$ is the imaginary unit. One can also refer to Lemma 2.4 in \cite{TxGL} or Lemma 2.3 in \cite{CzwZ}.  \\
\end{proof}

\section{proof of Proposition  \ref{prop1.4} and Corollary \ref{coro1.5}}
In this section, we give some basic properties of a partial weighted Morrey space and then prove Proposition  \ref{prop1.4} and Corollary \ref{coro1.5}.

The Morrey spaces were introduced by C. B. Morrey in 1938 \cite{CBMO} to investigate the local behavior of solutions
to some partial differential equations. Nowadays the Morrey spaces were extended to more general cases(see \cite{GPAP,YKSS,YSAW}). Let $p\in[1,+\infty)$ and $\gamma\in (0,n)$, the usual homogeneous Morrey space
$$L^{p,\gamma}(\R^{n})=\Big \{ u:  ||u||_{  L^{p,\gamma}(\R^{n}) } < +\infty \Big \}$$
was introduced in \cite{GPAP} with the norm
$$ ||u||_{  L^{p,\gamma}(\R^{n}) }=\mathop {\sup }\limits_{R>0,x \in \R^{n}} \Big \{R^{\gamma-n} \int_{B_R(x)} |u(y)|^pdy  \Big \}^{\frac{1}{p}} .$$
One can see that if $\gamma=n$ then $L^{p,\gamma}(\R^{n})$ coincide with $L^p(\R^{n})$ for any $p\geq1$; Similarly $L^{p,0}(\R^{n})$ coincide with $L^{\infty}(\R^{n})$.

Now we introduce a partial weighted Morrey space $L^{p,\gamma +\lambda}(\R^{n},|y'|^{-\lambda})$, which was motivated by  \cite{YKSS,YSAW,gBtY}. For $p\in[1,+\infty)$, $\gamma, \lambda>0$,  $\gamma+\lambda \in (0,n)$ and $0< m<n$, we say a Lebesgue measurable function $u: \R^{n}\rightarrow \R$ belongs to $L^{p,\gamma +\lambda}(\R^{n},|y'|^{-\lambda})$ if
$$ ||u||_{  L^{p,\gamma +\lambda}(\R^{n},|y'|^{-\lambda}) }=\mathop {\sup }\limits_{R>0,x \in \R^{n}} \Big \{R^{\gamma+\lambda-n} \int_{B_R(x)} \frac{|u(y)|^p}{  |y'|^{\lambda} }dy  \Big \}^{\frac{1}{p}} < +\infty,$$
where $y\!=\!(y',y'') \in \mathbb{R}^{m} \times \mathbb{R}^{n-m}$.

Then the following fundamental properties \textbf{(1)}-\textbf{(6)} hold via H\"{o}lder's inequality:

\textbf{(1)}~$L^{p\rho}(\R^{n},|y'|^{-\rho\lambda}) \hookrightarrow L^{p,\gamma +\lambda}(\R^{n},|y'|^{-\lambda})$ for $\rho=\frac{n}{\gamma +\lambda}>1$.

\textbf{(2)}~For any $p\in(1,+\infty)$, we have $L^{p,\gamma +\lambda}(\R^{n},|y'|^{-\lambda}) \hookrightarrow L^{1,\frac{\gamma}{p}+ \frac{\lambda}{p} }(\R^{n},|y'|^{-\frac{\lambda}{p}}) .$

\textbf{(3)}~Take $\gamma+\lambda=n$, we get $L^{p}(\R^{n},|y'|^{-\lambda})$.

\textbf{(4)}~Let $p\geq2$ and $u,v\in L^{p,\gamma +\lambda}(\R^{n},|y'|^{-\lambda})$, then $uv$ belongs to $L^{{p}/{2},\gamma +\lambda}(\R^{n},|y'|^{-\lambda})$ since
\begin{align*}
||(uv)||_{  L^{{p}/{2},\gamma +\lambda}(\R^{n},|y'|^{-\lambda}) }
\leq ||u||_{  L^{p,\gamma +\lambda}(\R^{n},|y'|^{-\lambda}) }  ||v||_{  L^{p,\gamma +\lambda}(\R^{n},|y'|^{-\lambda})} < +\infty.
\end{align*}
Moreover, if we assume $s\in (0,1)$, $0<\alpha<2s<n$ and $0< m<n$, then we have

\textbf{(5)}~For any $p\in[1,2^*_{s}(\alpha))$, ${\dot{H}}^s(\R^{n}) \hookrightarrow L^{2^*_{s}(\alpha)}(\R^{n},|y'|^{-\alpha}) \hookrightarrow L^{p,\frac{n-2s}{2}p+pr}(\R^{n},|y'|^{-pr})$ with $r=\frac{\alpha}{2^*_{s}(\alpha)}$ and the three norms in these spaces share the same dilation invariance. 

\textbf{(6)}~For any $p\in[1,2^*_{s})$, ${\dot{H}}^s(\R^{n}) \hookrightarrow L^{2^*_{s}}(\R^{n}) \hookrightarrow L^{p,\frac{n-2s}{2}p}(\R^{n})$, refer to page 815 in \cite{GPAP}. \\

%
%
%
%

\begin{lemma} (Theorem 1 in \cite{ESRW}, or Theorem D in \cite{BMRW}) \label{lemma3.4}
Suppose that $0<\tilde{s}<n$, $1 < \tilde{p} \leq \tilde{q} <+\infty$, $\tilde{p}'=\frac{\tilde{p}}{\tilde{p}-1}$ and that $V$ and $W$ are nonnegative measurable functions on $\R^{n}$, $n \geq 1$. If for some $\sigma>1$
\begin{equation}\label{eq3.3}
|Q|^{ \frac{\tilde{s}}{n}+\frac{1}{\tilde{q}}-\frac{1}{\tilde{p}}} { \Big(  \frac{1}{|Q|} \int_{Q}V^{\sigma}dy    \Big) }^{ \frac{1}{{\tilde{q}} \sigma} } { \Big(  \frac{1}{|Q|} \int_{Q}W^{(1-{\tilde{p}}')\sigma}dy    \Big) }^{ \frac{1}{ {\tilde{p}}' \sigma} } \leq C_{\sigma}
\end{equation}
for all cubes $Q \subset {\R}^n$, then for any function $f \in L^{\tilde{p}}({\R}^n,W(y))$ we have
\begin{equation}\label{eq3.4}
 { \Big(   \int_{{\R}^n}|{\ell}_{\tilde{s}}f(y)|^{\tilde{q}}V(y)dy    \Big) }^{ \frac{1}{\tilde{q}} }  \leq C C_{\sigma} { \Big(   \int_{{\R}^n}|f(y)|^{\tilde{p}}W(y)dy    \Big) }^{ \frac{1}{\tilde{p}} }
\end{equation}
where $C=C(\tilde{p},\tilde{q},n)$ and ${\ell}_{\tilde{s}}f(y)=\int_{{\R}^n} \frac{f(z)}{|y-z|^{n-{\tilde{s}}}}dz$ is the Riesz potential of order $\tilde{s}$. 
\end{lemma}

\noindent\textbf{Proof of Proposition \ref{prop1.4}}

For $u \in {\dot{H}}^s(\R^{n})$, we have $\hat{g}(\xi):=\!|\xi|^s\hat{u}(\xi) \!\in\! L^2(\R^{n})$ and
$||u||_{{\dot{H}}^s(\R^{n})}\!=\!||g||_{L^2(\R^{n})}$  by Plancherel's theorem. Thus, $u(x)\!=\!(\frac{1}{|\xi|^s})^{\vee}*g(x)\!=\!{\ell}_sg(x)$, where ${\ell}_sg(x)\!=\!\int_{{\R}^n} \frac{g(z)}{|x-z|^{n-{s}}}dz$. Similarly, let $v \!\in\! {\dot{H}}^s(\R^{n})$ and $\hat{f}(\xi):=\!|\xi|^s\hat{v}(\xi) \!\in\! L^2(\R^{n})$, we have $v(x)\!=\!(\frac{1}{|\xi|^s})^{\vee}*f(x)\!=\!{\ell}_sf(x)$.

Firstly, let $2s\!<\!m\!<\!n$ and take $\tilde{s}\!=\!s$, $\tilde{p}\!=\!2$, $\max \{ 2, 2^*_{s}(\alpha)-\frac{\alpha}{s}, \frac{2^*_{s}(\alpha)-\frac{\alpha}{s}}
{2^*_{s}(\alpha)-\frac{2\alpha}{m}}\cdot2^*_{s}(\alpha) \} \!<\! {\tilde{q}}\!<\! 2^*_{s}(\alpha)$, $W(y)\!\equiv \!1$, $V(y)\!=\!\frac{  {|(uv)(y)|}^  {\frac{ 2^{*}_{s}(\alpha)-{\tilde{q}}}{2} }   }   {  |y'|^{\alpha}  }$ and $\sigma\!=\!\frac{2s}{\alpha}\!>\!1$ in Lemma \ref{lemma3.4}, where $y\!=\!(y',y'') \!\in\! \mathbb{R}^{m} \!\times\! \mathbb{R}^{n-m}$. Then, (\ref{eq3.3}) becomes
\begin{equation}\label{eq3.5}
|Q|^{ \frac{s}{n}+\frac{1}{\tilde{q}}-\frac{1}{2}} { \Big(  \frac{1}{|Q|} \int_{Q}V^{\sigma}dy    \Big) }^{ \frac{1}{{\tilde{q}} \sigma} }  \leq C_{\sigma}.
\end{equation}

Secondly, we verify condition (\ref{eq3.3}). From $2^*_{s}(\alpha)\!-\!\frac{\alpha}{s}\!<\! {\tilde{q}}\!<\! 2^*_{s}(\alpha)$, we have $0\!<\!{\frac{[ 2^{*}_{s}(\alpha)-{\tilde{q}}]\sigma}{2} }
\!<\!1$. Moreover, we have $\frac{ t \sigma {\alpha}}{ 1- {\frac{[ 2^{*}_{s}(\alpha)-{\tilde{q}}]\sigma}{2} } }\!<\!m$ by $\tilde{q}\!>\!\frac{2^*_{s}(\alpha)-\frac{\alpha}{s}}
{2^*_{s}(\alpha)-\frac{2\alpha}{m}}\cdot2^*_{s}(\alpha)$, where $t:=\frac{\tilde{q}}{2^{*}_{s}(\alpha)}\in(0,1)$. For any fixed $x\in {\R}^n$, replacing $Q$ by ball $B_{R}(x)$, 
we deduce by H\"{o}lder's inequality that
\begin{align*}
R^{-n} &\int_{B_R(x)}V^{\sigma}dy=R^{-n} \int_{B_R(x)}{\frac{{|uv|}^{\frac{[ 2^{*}_{s}(\alpha)-{\tilde{q}}]\sigma}{2} }}{|y'|^  {  \sigma {\alpha}    }   }}dy=R^{-n} \int_{B_R(x)} \frac{1}{{|y'|^  { t \sigma {\alpha}    }   }  } \cdot {\frac{{|uv|}^{\frac{[ 2^{*}_{s}(\alpha)-{\tilde{q}}]\sigma}{2} }}{|y'|^  { (1-t) \sigma {\alpha}    }   }}dy  \\
&\leq R^{-n}  { \Big( \int_{B_R(x)} \frac{dy}{{|y'|^  {  \frac{ t \sigma {\alpha}}{ 1- {\frac{[ 2^{*}_{s}(\alpha)-{\tilde{q}}]\sigma}{2} } }       }   }  }  \Big ) }^{1- {\frac{[ 2^{*}_{s}(\alpha)-{\tilde{q}}]\sigma}{2} }   }   { \Big( \int_{B_R(x)}  {\frac{{|uv|}}{|y'|^r  }}dy  \Big ) }^{{\frac{[ 2^{*}_{s}(\alpha)-{\tilde{q}}]\sigma}{2} }  }   \\
&\leq R^{-n}  { \Big( \int_{|y''-x''|\leq R} dy''  \int_{|y'-x'|\leq R} \frac{dy'}{{|y'|^  {  \frac{ t \sigma {\alpha}}{ 1- {\frac{[ 2^{*}_{s}(\alpha)-{\tilde{q}}]\sigma}{2} } }       }   }  }  \Big ) }^{1- {\frac{[ 2^{*}_{s}(\alpha)-{\tilde{q}}]\sigma}{2} }   }   { \Big( \int_{B_R(x)}  {\frac{{|uv|}}{|y'|^r  }}dy  \Big ) }^{{\frac{[ 2^{*}_{s}(\alpha)-{\tilde{q}}]\sigma}{2} }  }   \\
&\leq R^{-n}  { \Big( \int_{B_R(x'')} dy''  \int_{|y'|\leq R} \frac{dy'}{{|y'|^  {  \frac{ t \sigma {\alpha}}{ 1- {\frac{[ 2^{*}_{s}(\alpha)-{\tilde{q}}]\sigma}{2} } }       }   }  }  \Big ) }^{1- {\frac{[ 2^{*}_{s}(\alpha)-{\tilde{q}}]\sigma}{2} }   }   { \Big( \int_{B_R(x)}  {\frac{{|uv|}}{|y'|^r  }}dy  \Big ) }^{{\frac{[ 2^{*}_{s}(\alpha)-{\tilde{q}}]\sigma}{2} }  }   \\
&\leq CR^{-t\alpha \sigma-\frac{n[{2^{*}_{s}(\alpha)}-{\tilde{q}}]\sigma}{2}}   { \Big( \int_{B_R(x)}  {\frac{{|uv|}}{|y'|^r }}dy  \Big ) }^{\frac{[ 2^{*}_{s}(\alpha)-{\tilde{q}}]\sigma}{2} },
\end{align*}
where $r:={\frac{2(1-t) {\alpha}}{{2^{*}_{s}(\alpha)}-{\tilde{q}}} }=\frac{2\alpha}{2^{*}_{s}(\alpha)}$. Therefore,
\begin{align*}
&R^{s+\frac{n}{{\tilde{q}}}-\frac{n}{2}} { \Big(  R^{-n} \int_{ B_R(x) }V^{\sigma}dy    \Big) }^{ \frac{1}{{\tilde{q}} \sigma} }  \\
\leq & R^{s+\frac{n}{{\tilde{q}}}-\frac{n}{2}} \Big\{ CR^{-t\alpha \sigma-\frac{n[{2^{*}_{s}(\alpha)}-{\tilde{q}}]\sigma}{2}}   { \Big( \int_{B_R(x)}  {\frac{{|uv|}}{|y'|^r }}dy  \Big ) }^{\frac{[ 2^{*}_{s}(\alpha)-{\tilde{q}}]\sigma}{2} }  \Big\}^{\frac{1}{{\tilde{q}} \sigma}}  \\
\leq & C \Big\{  R^{  (s+\frac{n-t\alpha}{{\tilde{q}}}-\frac{n}{2})
{\frac{2{\tilde{q}}}{{2^{*}_{s}(\alpha)}-{\tilde{q}}}}   }  \cdot R^{-n}  \cdot {  \int_{B_R(x)}  {\frac{{|uv|}}{|y'|^r }}dy   }   \Big\}^{  \frac{{2^{*}_{s}(\alpha)}-{\tilde{q}}}{2\tilde{q}}  } \\
= & C \Big\{  R^{   n-2s+r }  \cdot R^{-n}  \cdot {  \int_{B_R(x)}  {\frac{{|uv|}}{|y'|^  { r }   }}dy   }   \Big\}^{  \frac{{2^{*}_{s}(\alpha)}-{\tilde{q}}}{2\tilde{q}}  }
\leq  C {||(uv)||}^{\frac{{2^{*}_{s}(\alpha)}-{\tilde{q}}}
{2\tilde{q}}}_{ L^{1,n-2s+r }(\R^{n},|y'|^{-r})}:=C_{\sigma}.
\end{align*}
Since $u={\ell}_sg$ and $v={\ell}_sf$, and by Lemma \ref{lemma3.4},
\begin{align*}
&\int_{ \R^{n} }  \frac{ |(uv)(y)|^{\frac{2^*_{s}(\alpha)}{2} } }  {  |y'|^{\alpha} } dy = \int_{ \R^{n} } V(y) {|{\ell}_sg(y)|}^{\frac{\tilde{q}}{2}}{|{\ell}_sf(y)|}^{\frac{\tilde{q}}{2}} dy \\
&\leq \Big[\int_{ \R^{n} } V(y) {|{\ell}_sg(y)|}^{\tilde{q}} dy\Big]^{\frac{1}{2}} \Big[\int_{ \R^{n} } V(y) {|{\ell}_sf(y)|}^{\tilde{q}} dy\Big]^{\frac{1}{2}}   \\
& \leq (CC_{\sigma})^{\tilde{q}} ||g||^{\frac{\tilde{q}}{2}}_{L^2} ||f||^{\frac{\tilde{q}}{2}}_{L^2}   \leq C ||u||_{{\dot{H}}^s(\R^{n})}^{\frac{\tilde{q}}{2}}
 ||v||_{{\dot{H}}^s(\R^{n})}^{\frac{\tilde{q}}{2}} ||(uv)||^{\frac{2^*_{s}(\alpha)-\tilde{q}}{2}}_{  L^{1,n-2s+r}(\R^{n},|y'|^{-r}) }.
\end{align*}
Then, for any $\theta=\frac{{\tilde{q}}}{ 2^*_{s}(\alpha) }$ satisfying $\bar{\theta}<\theta<1$, we have
\begin{equation*}
 \Big( \int_{ \R^{n} }  \frac{ |(uv)(y)|^{\frac{2^*_{s}(\alpha)}{2} } }  {  |y'|^{\alpha} } dy  \Big)^{ \frac{1}{  2^*_{s} (\alpha)  }}  \leq C ||u||_{{\dot{H}}^s(\R^{n})}^{\frac{\theta}{2}}
 ||v||_{{\dot{H}}^s(\R^{n})}^{\frac{\theta}{2}} ||(uv)||^{\frac{1-\theta}{2}}_{  L^{1,n-2s+r}(\R^{n},|y'|^{-r}) },
\end{equation*}
where $s \in (0,1)$, $0<\alpha<2s<m<n$, $y\!=\!(y',y'') \in \mathbb{R}^{m} \times \mathbb{R}^{n-m}$, $r=\frac{2\alpha}{ 2^*_{s}(\alpha) }$, $\bar{\theta}=\max \{ \frac{2}{2^*_{s}(\alpha)}, 1-\frac{\alpha}{s}\cdot\frac{1}{2^*_{s}(\alpha)}, \frac{2^*_{s}(\alpha)-\frac{\alpha}{s}}{2^*_{s}(\alpha)-\frac{2\alpha}{m}} \}$ and $C=C(n,m,s,\alpha)>0$ is a constant. \qed  \\

\noindent\textbf{Proof of Corollary \ref{coro1.5}}

For $n\!\geq\!3$ and $u \!\in\! {C}_{0}^{\infty}(\R^{n})$, we have $u(x)\!=\!{\Delta}^{-1}\Delta u\!=\!C_1\int_{\R^{n}}\frac{\Delta u(y)}{{|x-y|}^{n-2}}dy\!=\!C_2\int_{\R^{n}}\frac{(x-y)\nabla u(y)}{{|x-y|}^n}dy$.
Thus, $|u(x)| \!\leq\! |{C_2}|\int_{\R^{n}}\frac{| {\nabla u(y)} |}{{|x-y|}^{n-1}}dy \!\leq\! C {\ell}_{1}(| {\nabla u} |)(x)$,
where $C_1,C_2,C$ are positive constants (depending on $n$). These inequalities hold for $n=2$ via the logarithmic kernel(See \cite{GPAP}). By density of ${C}_{0}^{\infty}(\R^{n})$ in ${D}^{1,p}(\R^{n})$, it is also true for any $u \in {D}^{1,p}(\R^{n})(n\geq 2)$.

Let $2\!\leq\! p\!<\!m\!<\!n$ and take $\tilde{s}\!=\!1$, $\tilde{p}\!=\!p$, $\max \{ p, p^*(\alpha)-\alpha,\frac{p^*(\alpha)
-\alpha}{p^*(\alpha)-\frac{p\alpha}{m}}\cdot p^*(\alpha) \} \!<\! {\tilde{q}}\!< \!p^*(\alpha)$, $W(y)\!\equiv\! 1$, $V(y)\!=\!\frac{  {|(uv)(y)|}^  {\frac{ p^*(\alpha)-{\tilde{q}}}{2} }   }   {  |y'|^{\alpha}  }$ and $\sigma\!=\!\frac{p}{\alpha}\!>\!1$ in Lemma \ref{lemma3.4}, where $y\!=\!(y',y'') \in \mathbb{R}^{m} \times \mathbb{R}^{n-m}$. The remain argument is similar to that in $ {\dot{H}}^s(\R^{n})$ with $t:=\frac{\tilde{q}}{p^{*}(\alpha)}$, $r:={\frac{p(1-t) {\alpha}}{{p^{*}(\alpha)}-{\tilde{q}}} }=\frac{p\alpha}{p^{*}(\alpha)}$, $\theta=\frac{{\tilde{q}}}{ p^*(\alpha) }\in(\bar{\theta},1)$ and $\bar{\theta}=\max \{ \frac{p}{p^*(\alpha)}, 1-\frac{\alpha}{p^*(\alpha)},\frac{p^*(\alpha)
-\alpha}{p^*(\alpha)-\frac{p\alpha}{m}} \}$. Actually, we have
\begin{equation*}
 \Big( \int_{ \R^{n} }  \frac{  {|(uv)|}^  {\frac{ p^*(\alpha)}{2} }   }   {  |y'|^{\alpha}  } dy  \Big)^{ \frac{1}{  p^* (\alpha)  }}  \leq C ||u||_{{D}^{1,p}(\R^{n})}^{\frac{\theta}{2}}
 ||v||_{{D}^{1,p}(\R^{n})}^{\frac{\theta}{2}} ||(uv)||^{\frac{1-\theta}{2}}_{  L^{{p}/{2},n-p+r}(\R^{n},|y'|^{-r}) }.
\end{equation*}


\begin{lemma} (Theorem 1 in \cite{GPAP}) \label{lemma3.6}
Let $s \in (0,1)$, $n>2s$ and $ 2^*_{s}=\frac{2n}{n-2s} $. Then there exists $C=C(n,s)>0$ such that for any $ \max\{ \frac{2}{2^*_{s}},1-\frac{1}{2^*_{s}} \} < \theta <1 $ and for any $1\leq p <2^*_{s}$
\begin{equation}\label{eq3.6}
 ||u||_{ L^{ 2^*_{s} }(\R^{n})} \leq C ||u||_{{\dot{H}}^s(\R^{n})}^{\theta} ||u||^{1-\theta}_{  L^{p,\frac{n-2s}{2}p}(\R^{n}) },~~~~ \forall u \in {\dot{H}}^s(\R^{n}).
\end{equation}
\end{lemma}

\begin{remark}
If $u=v$ and $\alpha=0$, then inequality (\ref{eqa1.6}) becomes (\ref{eq3.6}) with $p=2$.
\end{remark}

\section{Solving the minimization problems (\ref{eq1.9})-(\ref{eq1.8})}
In this section, we solve the minimization problems (\ref{eq1.9})-(\ref{eq1.8}). Using the embeddings (\ref{eq1.06}) and the inequality \eqref{eqa1.6}, we can prove the existence of minimizers for
\begin{equation*}
  {\Lambda}(n,s,\alpha)=  \mathop {\inf }\limits_{(u,v) \in X \setminus \{(0,0)\}  }   \frac{ {||(u,v)||}^2 }{\Big( \int_{\R^{n}} \frac{{|uv|}^{   \frac{2^{*}_{s}(\alpha)}{2}     }}{|y'|^{\alpha}}dy\Big)^{\frac{2}{  2^{*}_{s}(\alpha)  }}},
\end{equation*}
and
\begin{equation*}
   {S}(n,s,\alpha)=\mathop {\inf }\limits_{(u,v) \in X \setminus \{(0,0)\}  }    \frac{ {||(u,v)||}^2 }{ \Big( \int_{\R^{n}} \frac{{|u|}^{ { 2^{*}_{s} }(\alpha)}}{|y'|^{\alpha}}dy+\int_{\R^{n}} \frac{{|u|}^{ { 2^{*}_{s} }(\alpha)}}{|y'|^{\alpha}}dy   \Big)^{\frac{2}{ 2_s^*(\alpha)}} },
\end{equation*}
where $X=\dot{H}^{s}(\R^{n})\times\dot{H}^{s}(\R^{n})$ and $||(u,v)||^2=||u||_{\gamma_1}^2+||v||_{\gamma_2}^2$ was defined in  \eqref{DeF2.4}.

\begin{proposition}\label{pro1.7}
Let $s \in(0,1)$ and $X=\dot{H}^{s}(\R^{n})\times\dot{H}^{s}(\R^{n})$. Then  \\
$(1)$ If $0\!<\!\alpha\!<\!2s\!<\!n$, $2s\!<\!m\!<\!n$ and $\gamma_1,\gamma_2\!<\!\gamma_{H}$, then ${\Lambda}(n,s,\alpha)$ is attained in $X \!\setminus\! \{(0,0)\}$; \\
$(2)$ If $0\!<\!\alpha\!<\!2s\!<\!n$, $2s\!<\!m\!<\!n$ and $\gamma_1,\gamma_2\!<\!\gamma_{H}$, then $S(n,s,\alpha)$ is attained in $X \!\setminus\! \{(0,0)\}$; \\
$(3)$ If $n\!>\!m\!>\!2s$ and $0\!\leq\!\gamma_1,\gamma_2\!<\!\gamma_{H}$, then ${S}(n,s,0)$ is attained in $X \!\setminus \!\{(0,0)\}$.\\
\end{proposition}

\noindent\textbf{Proof of Proposition \ref{pro1.7} }

$\textbf{(1)}$ If $ 0<\alpha<2s<m <n $ and $\gamma_1,\gamma_2<\gamma_H$, let $\{(u_k,v_k)\}$ be a minimizing sequence of ${\Lambda}(n,s,\alpha)$, that is
$$   \int_{\R^{n}} \frac{{|(u_kv_k)(y)|}^{   \frac{2^{*}_{s}(\alpha)}{2}     }}{|y'|^{\alpha}}dy=1,~~~~~~~~{||(u_k,v_k)||}^2 \rightarrow {\Lambda}(n,s,\alpha),$$
where $y\!=\!(y',y'') \in \mathbb{R}^{m} \times \mathbb{R}^{n-m}$. Then the embeddings (\ref{eq1.06}) and the improved Sobolev inequality \eqref{eqa1.6} imply that there exists $C>0$ such that
$$   0<C \leq ||(u_kv_k)||_{  L^{1,n-2s+r}(\R^{n},|y'|^{-r}) }
\leq ||u_k||_{  L^{2,n-2s+r}(\R^{n},|y'|^{-r}) } ||v_k||_{  L^{2,n-2s+r}(\R^{n},|y'|^{-r}) }\leq C^{-1},$$
where $r\!=\!\frac{2\alpha}{ 2^*_{s}(\alpha) }$. For any $k\!\geq\!1$, we may find ${\lambda}_k\!>\!0$ and $x_k\!=\!(x'_k,x''_k) \!\in\! \R^{m}\!\times\! \R^{n-m}$ such that
$$  {\lambda}_k^{-2s+r} \int_{B_{{\lambda}_k}(x_k)} \frac{|(u_kv_k)(y)|}{  |y'|^{r} }dy > ||(u_kv_k)||_{  L^{1,n-2s+r}(\R^{n},|y'|^{-r}) } -\frac{C}{2k} \geq C_1>0,$$
where $y\!=\!(y',y'') \in \mathbb{R}^{m} \times \mathbb{R}^{n-m}$.

Let $\tilde{u}_k(x)={\lambda}_k^{ \frac{n-2s}{2} }u_k\big({\lambda}_kx', {\lambda}_k(x''-x''_k)\big)$, $\tilde{v}_k(x)={\lambda}_k^{ \frac{n-2s}{2} }v_k\big({\lambda}_kx', {\lambda}_k(x''-x''_k)\big)$ and ${\tilde{x}}_k=\frac{x_k}{{\lambda}_k}
=({\tilde{x}}'_k,{\tilde{x}}''_k)\in \R^{m}\times \R^{n-m}$, then
\begin{equation}\label{Eeq4.1}
\int_{B_{1}(({\tilde{x}}'_k,\vec{0}))} \frac{|(\tilde{u}_k \tilde{v}_k)(x)|}{  |x'|^{r} }dx  \geq C_1>0,
\end{equation}
where $x\!=\!(x',x'') \!\in\! \mathbb{R}^{m} \!\times\! \mathbb{R}^{n-m}$ and $\vec{0}\!=\!(0,\cdots,0)\!\in\! \R^{n-m}$. We claim that $\{  {\tilde{x}}'_k \}$ is bounded. Otherwise, $|{\tilde{x}}'_k| \rightarrow  +\infty$, then for any $x \in B_{1}(({\tilde{x}}'_k,\vec{0}))$, $|x'| \geq |{\tilde{x}}'_k|-1$ for $k$ large. Therefore,
\begin{align*}
\int_{B_{1}(({\tilde{x}}'_k,\vec{0}))} \frac{|\tilde{u}_k\tilde{v}_k(x)|}{  |x'|^{r} }dx  &\leq    \frac{1}{  (|{\tilde{x}}'_k|-1)^{r} }   \int_{B_{1}(({\tilde{x}}'_k,\vec{0}))}   |\tilde{u}_k\tilde{v}_k(x)|    dx  \\
&\leq       \frac{{ \Big( \int_{B_{1}(({\tilde{x}}'_k,\vec{0}))}dx\Big)}^{1-\frac{2}{2^*_{s}}}}{  (|{\tilde{x}}'_k|-1)^{r} }   { \Big( \int_{B_{1}(({\tilde{x}}'_k,\vec{0}))}   |\tilde{u}_k(x)|^{2^*_{s}}    dx  \Big) }^{\frac{1}{2^*_{s}}} { \Big( \int_{B_{1}(({\tilde{x}}'_k,\vec{0}))}   |\tilde{v}_k(x)|^{2^*_{s}}    dx  \Big) }^{\frac{1}{2^*_{s}}}\\
&\leq   \frac{C}{  (|{\tilde{x}}'_k|-1)^{r} } ||\tilde{u}_k||_{{\dot{H}}^s(\R^{n})} ||\tilde{v}_k||_{{\dot{H}}^s(\R^{n})} \leq       \frac{\tilde{C}}{  (|{\tilde{x}}'_k|-1)^{r} }   \rightarrow 0 
\end{align*}
as $k \rightarrow  +\infty$, which contradicts to (\ref{Eeq4.1}). From (\ref{Eeq4.1}), we may find $R>0$ such that
\begin{equation}\label{Eeq4.3}
  \int_{B_{R}(0)}  \frac{|(\tilde{u}_k \tilde{v}_k)(x)|}{  |x'|^{r} }dx  \geq \int_{B_{1}(({\tilde{x}}'_k,\vec{0}))} \frac{|(\tilde{u}_k \tilde{v}_k)(x)|}{  |x'|^{r} }dx \geq C_1>0.
\end{equation}

Since ${\Lambda}(n,s,\alpha)$ is invariant under the previous dilation and partial translation given by $\lambda_k$ and $x''_k$, we have
$$   \int_{\R^{n}} \frac{{|\tilde{u}_k\tilde{v}_k|}^{   \frac{2^{*}_{s}(\alpha)}{2}     }}{|x'|^{\alpha}}dx=1,~~~~~~~~{||(\tilde{u}_k,\tilde{v}_k)||}^2 \rightarrow {\Lambda}(n,s,\alpha).$$
By the fact that ${||(\tilde{u}_k,\tilde{v}_k)||}={||({u}_k,{v}_k)||}\leq C$, there exists $(\tilde{u},\tilde{v})\in X$ such that
\begin{align*} 
  (\tilde{u}_k,\tilde{v}_k) \rightharpoonup (\tilde{u},\tilde{v}) ~~\mbox{in} ~~X,~~~~~~ (\tilde{u}_k,\tilde{v}_k) \rightarrow (\tilde{u},\tilde{v}) ~~\mbox{a.e. ~~~~on}~~ \R^n\times \R^n
\end{align*}
up to subsequences. According to Lemma \ref{lemma2.3}, we have
$$\frac{|\tilde{u}_k|}{ {|x'|}^{\frac{r}{2}}  }    \rightarrow   \frac{|\tilde{u}|}{ {|x'|}^{\frac{r}{2}}  } ~~\mbox{in}~~L^2_{loc}(\R^{n}),~~~~~~~~\frac{|\tilde{v}_k|}{ {|x'|}^{\frac{r}{2}}  }    \rightarrow   \frac{|\tilde{v}|}{ {|x'|}^{\frac{r}{2}}  } ~~\mbox{in}~~L^2_{loc}(\R^{n})$$
since $\frac{r}{2}=\frac{\alpha}{ 2^*_{s}(\alpha) }<s$. Therefore, (\ref{Eeq4.3}) leads to
$$  \int_{B_{R}(0)}  \frac{|(\tilde{u} \tilde{v})(x)|}{  |x'|^{r} }dx  \geq C_1>0, $$
and we deduce that $\tilde{u}\not \equiv 0$ and $\tilde{v}\not \equiv 0$. We may verify as Lemma \ref{LmMa2.5} that
\begin{align*}
1=\int_{\R^{n}} \frac{{|\tilde{u}_k\tilde{v}_k|}^{   \frac{2^{*}_{s}(\alpha)}{2}     }}{|x'|^{\alpha}}dx=\int_{\R^{n}} \frac{{|(\tilde{u}_k-\tilde{u})(\tilde{v}_k-\tilde{v})|}^{   \frac{2^{*}_{s}(\alpha)}{2}     }}{|x'|^{\alpha}}dx+\int_{\R^{n}} \frac{{|\tilde{u}\tilde{v}|}^{   \frac{2^{*}_{s}(\alpha)}{2}     }}{|x'|^{\alpha}}dx+o(1).
\end{align*}
By the weak convergence $\tilde{u}_k \rightharpoonup \tilde{u}$ in ${\dot{H}}^s(\R^{n})$ and $\tilde{v}_k \rightharpoonup \tilde{v}$ in ${\dot{H}}^s(\R^{n})$, we have
\begin{align*}
     &{\Lambda}(n,s,\alpha)=\mathop {\lim }\limits_{k  \to \infty} {||(\tilde{u}_k,\tilde{v}_k)||}^2=  {||(\tilde{u},\tilde{v})||}^2+ \mathop {\lim }\limits_{k  \to \infty} {||(\tilde{u}_k-\tilde{u},\tilde{v}_k-\tilde{v})||}^2   \\
     &\geq {\Lambda}(n,s,\alpha)  \Big( \int_{\R^{n}} \frac{{|\tilde{u}\tilde{v}|}^{   \frac{2^{*}_{s}(\alpha)}{2}     }}{|x'|^{\alpha}}dx  \Big)^{\frac{2}{  2^{*}_{s}(\alpha)  }}
     +{\Lambda}(n,s,\alpha) \mathop {\lim }\limits_{k  \to \infty} \Big( \int_{\R^{n}} \frac{{|(\tilde{u}_k-\tilde{u})(\tilde{v}_k-\tilde{v})|}^{   \frac{2^{*}_{s}(\alpha)}{2}     }}{|x'|^{\alpha}}dx  \Big)^{\frac{2}{  2^{*}_{s}(\alpha)  }} \\
     &\geq {\Lambda}(n,s,\alpha)  \Big(\int_{\R^{n}} \frac{{|\tilde{u}\tilde{v}|}^{   \frac{2^{*}_{s}(\alpha)}{2}     }}{|x'|^{\alpha}}dx
     + \mathop {\lim }\limits_{k  \to \infty} \int_{\R^{n}} \frac{{|(\tilde{u}_k-\tilde{u})(\tilde{v}_k-\tilde{v})|}^{   \frac{2^{*}_{s}(\alpha)}{2}     }}{|x'|^{\alpha}}dx  \Big)^{ \frac{2}{2^{*}_{s}(\alpha)}}={\Lambda}(n,s,\alpha).
\end{align*}
Here we use the fact that $(a+b)^{\frac{2}{2^{*}_{s}(\alpha)}} \leq a^{\frac{2}{2^{*}_{s}(\alpha)}}+b^{\frac{2}{2^{*}_{s}(\alpha)}}$, $\forall a\geq0, b\geq0$ and $\frac{2}{2^{*}_{s}(\alpha)}<1.$\\
So we have
$$\int_{\R^{n}} \frac{{|\tilde{u}\tilde{v}|}^{   \frac{2^{*}_{s}(\alpha)}{2}     }}{|x'|^{\alpha}}dx=1,~~~~~~\mathop {\lim }\limits_{k  \to \infty}\int_{\R^{n}} \frac{{|(\tilde{u}_k-\tilde{u})(\tilde{v}_k-\tilde{v})|}^{   \frac{2^{*}_{s}(\alpha)}{2}     }}{|x'|^{\alpha}}dx=0,$$
since $\tilde{u} \not \equiv 0$ and $\tilde{v} \not \equiv 0$. It results to ${||(\tilde{u},\tilde{v})||}^2\leq {\Lambda}(n,s,\alpha)\leq {||(\tilde{u},\tilde{v})||}^2$ and so we have
$$ {\Lambda}(n,s,\alpha)=  {||(\tilde{u},\tilde{v})||}^2,~~~~ \mathop {\lim }\limits_{k  \to \infty} {||(\tilde{u}_k-\tilde{u},\tilde{v}_k-\tilde{v})||}^2=0.$$
By formula (A.11) in \cite{RSER}, we know that $\int_{\R^{n}} |(-\Delta)^{\frac{s}{2} }|\tilde{u}||^2 dx \leq  \int_{\R^{n}} |(-\Delta)^{\frac{s}{2} }\tilde{u}|^2 dx$. Hence, $(|\tilde{u}|,|\tilde{v}|)$ is also a minimizer of ${\Lambda}(n,s,\alpha)$, we can assume $\tilde{u} \geq 0$ and $\tilde{v} \geq 0$. 

$\textbf{(2)}$ If $0<\alpha<2s$ and $\gamma_1,\gamma_2<\gamma_H $, the proof is similar to that of Proposition \ref{pro1.7}-(1). In fact, this case is easier than Proposition \ref{pro1.7}-(1) because there is no coupled term in the right hand side of $S(n,s,\gamma,\alpha)$. We shall use the improved Sobolev inequality \eqref{eqa1.6} with $u=v$.

$\textbf{(3)}$ If $\alpha=0$ and $0\leq\gamma_1,\gamma_2<\gamma_H$, we combine the symmetric decreasing rearrangement technique with inequality \eqref{eq3.6}. One can refer to \cite{RFPP,SDLM,XcTy,gBtY}. \qed


\section{proof of  Theorem \ref{th1.1}}
We shall now use the minimizers of $S(n,s,\beta)$ and $\Lambda(n,s,\alpha)$ obtained in Proposition \ref{pro1.7}, to prove the existence of a nontrivial weak solution for system (\ref{eq1.1}). Recall that, for $(u,v) \in X:=\dot{H}^{s}(\R^{n})\times\dot{H}^{s}(\R^{n})$, the energy functional associated to (\ref{eq1.1}) is :
\begin{align}\label{eq5.1}
 I(u,v)&=\frac{1}{2}{||(u,v)||}^2-\frac{2}{2^{*}_{s}(\alpha)}\int_{\R^n}{\frac{{|uv|}^{ \frac{{2^{*}_{s}}(\alpha)}{2} }}{|x'|^{\alpha}}}dx-\frac{1}{ 2^{*}_{s}(\beta) }\int_{\R^{n}}{\frac{{|u|}^{ {2^{*}_{s}}(\beta)}+{|v|}^{ {2^{*}_{s}}(\beta)}}{|x'|^{\beta}}}dx,
\end{align}
where ${||(u,v)||}^2={||u||}_{\gamma_1}^2+{||v||}_{\gamma_2}^2$ was defined in (\ref{DeF2.4}). Fractional Sobolev and Hardy-Sobolev inequalities yield that $I \in C^1(X,\R)$ such that for any $(\phi,\psi)\in X$
\begin{align*}
{\langle I'(u,v),(\phi,\psi) \rangle}_X&={\langle u,\phi\rangle}_{\gamma_1}+{\langle v,\psi\rangle}_{\gamma_2}-\int_{\R^{n}} {\frac{{|u|}^{\frac{{2^{*}_{s}}(\alpha)}{2} -2}u\phi {|v|}^{\frac{{2^{*}_{s}}(\alpha)}{2}}+{|v|}^{\frac{{2^{*}_{s}}(\alpha)}{2} -2}v\psi {|u|}^{\frac{{2^{*}_{s}}(\alpha)}{2}}}{|x'|^{\alpha}}}dx\\
&-\int_{\R^{n}}  {\frac{{|u|}^{ {2^{*}_{s}}(\beta)-2}u\phi+{|v|}^{ {2^{*}_{s}}(\beta)-2}v\psi}{|x'|^{\beta}}}dx,
\end{align*}
where ${\langle ,\rangle}_X$ was defined in (\ref{DeF2.4}). 

\begin{lemma}(Mountain pass lemma, \cite{AMPH})  \label{lemma5.1}
Let $(E,||\cdot||)$ be a Banach space and $I\in  C^1(E,\R) $ satisfying the following conditions:  \\
$(1)$  $I(0)=0 ,$ \\
$(2)$ There exist $\rho,r>0 $ such that $I(z)\geq \rho $ for all $z \in E$ with $||z||=r,$  \\
$(3)$ There exist $v_0 \in E$ such that $  \lim_{t \to +\infty} {\sup }I(tv_0)<0.$  \\
Let $t_0>0$ be such that $||t_0v_0||>r$ and $I(t_0v_0)<0$, and define
$$  c:=\mathop {\inf }\limits_{g\in \Gamma}   \mathop {\sup }\limits_{t\in [0,1]} I(g(t)),$$
where $\Gamma:=\Big \{  g \in C^0([0,1],E) :g(0)=0,g(1)=t_0v_0   \Big \}$.
Then, $c\geq \rho >0$ and there exists a $(PS)$ sequence $\{z_k\} \subset E$ for $I$ at level $c$, i.e.
$$\lim_{k \to +\infty}I(z_k)=c~~\mbox{and}~~ \lim_{k \to +\infty} I'(z_k)=0~~\mbox{strongly in}~~E'.$$\\
\end{lemma}

We now use Lemma \ref{lemma5.1} to prove the following Proposition.

\begin{proposition}  \label{prop5.2}
Let $s \in(0,1)$, $0<\alpha,\beta<2s<n $, $2s<m<n $ and $\gamma_1,\gamma_2<\gamma_{H}$. Consider the functional $I$ defined in $(\ref{eq5.1})$ on the Banach space $X:=\dot{H}^{s}(\R^{n})\times\dot{H}^{s}(\R^{n})$. Then there exists a $(PS)$ sequence $\{(u_k,v_k)\} \subset X$ for $I$ at some $c\in (0,c^*)$, i.e.
\begin{equation}\label{eq5.2}
\lim_{k \to +\infty}I(u_k,v_k)=c~~\mbox{and}~~ \lim_{k \to +\infty} I'(u_k,v_k)=0~~\mbox{strongly in}~~X',
\end{equation}
where $c^*:=\min \Big \{ \frac{2s-\beta}{2(n-\beta)} {S(n,s,\beta)}^{\frac{n-\beta}{2s-\beta}}, \frac{2s-\alpha}{n-\alpha} \Big[\frac{\Lambda(n,s,\alpha)}{2}\Big]^{\frac{n-\alpha}{2s-\alpha}} \Big \}$.
\end{proposition}

\begin{proof}
We now verify the conditions of Lemma \ref{lemma5.1}. For any $(u,v)\in X$,
\begin{align*}
  I(u,v)&=\frac{1}{2}{||(u,v)||}^2-\frac{2}{2^{*}_{s}(\alpha)}\int_{\R^n}{\frac{{|uv|}^{ \frac{{2^{*}_{s}}(\alpha)}{2} }}{|x'|^{\alpha}}}dx
-\frac{1}{ 2^{*}_{s}(\beta) }\int_{\R^{n}}{\frac{{|u|}^{ {2^{*}_{s}}(\beta)}+{|v|}^{ {2^{*}_{s}}(\beta)}}{|x'|^{\beta}}}dx\\
  &\geq\frac{1}{2}||(u,v)||^2-C_1||(u,v)||^{2^{*}_{s}(\alpha)}-C_2||(u,v)||^{  2^{*}_{s}(\beta) }.
\end{align*}
Since $s \in(0,1)$ and $0<\alpha,\beta<2s< n $, we have that $2^{*}_{s}(\beta)>2$ and $2^{*}_{s}(\alpha)>2$. Therefore, there exists $r>0$ small enough such that
$$  \mathop {\inf }\limits_{||(u,v)||=r} I(u,v)>0=I(0,0),     $$
so $(1)$ and $(2)$ of Lemma \ref{lemma5.1} are satisfied. From
\begin{align*}
I\big(t(u,v)\big)=\frac{t^2}{2}||(u,v)||^2-\frac{2t^ {2^{*}_{s}(\alpha)}}{2^{*}_{s}(\alpha)}\int_{\R^n}{\frac{{|uv|}^{ \frac{{2^{*}_{s}}(\alpha)}{2} }}{|x'|^{\alpha}}}dx-\frac{t^ {2^{*}_{s}(\beta)}}{ 2^{*}_{s}(\beta) }\int_{\R^{n}}{\frac{{|u|}^{ {2^{*}_{s}}(\beta)}+{|v|}^{ {2^{*}_{s}}(\beta)}}{|x'|^{\beta}}}dx,
\end{align*}
we derive that $\lim_{t \to +\infty} I\big(t(u,v)\big)=-\infty$ for any $ (u,v)\ in ~~X$. Consequently, for any fixed $V_0=(v_0,\tilde{v}_0)\in X$, there exists ${t_{V_0}}>0$ such that $||t_{V_0} V_0||=||t_{V_0} (v_0,\tilde{v}_0)||>r$ and $I(t_{V_0} V_0)<0$. So $(3)$ of Lemma \ref{lemma5.1} is satisfied.

Using (1) and (2) in Proposition \ref{pro1.7}, we obtain a minimizer $U_{\beta}=(u_{\beta},\tilde{u}_{\beta}) \in X$ for $S(n,s,\beta)$ and $V_{\alpha}=(v_{\alpha},\tilde{v}_{\alpha}) \in X$ for $\Lambda(n,s,\alpha)$ respectively. So there exist
$$
V_0:=\left \{ \begin{array}{ll} U_{\beta},~~~~\mbox{if}&  \frac{2s-\beta}{2(n-\beta)} {S(n,s,\beta)}^{\frac{n-\beta}{2s-\beta}} \leq \frac{2s-\alpha}{n-\alpha} \Big[\frac{\Lambda(n,s,\alpha)}{2}\Big]^{\frac{n-\alpha}{2s-\alpha}} ; \\ V_{\alpha} ,~~~~\mbox{if}&  \frac{2s-\beta}{2(n-\beta)} {S(n,s,\beta)}^{\frac{n-\beta}{2s-\beta}}> \frac{2s-\alpha}{n-\alpha} \Big[\frac{\Lambda(n,s,\alpha)}{2}\Big]^{\frac{n-\alpha}{2s-\alpha}}
\end{array}    \right .   $$
and $t_0>0$ such that $||t_0V_0||>r$ and $I(t_0V_0)<0$. We can define
$$  c:=\mathop {\inf }\limits_{g\in \Gamma}   \mathop {\sup }\limits_{t\in [0,1]} I(g(t))$$
where $\Gamma:=\Big \{  g(t)=\big(g_1(t),g_2(t)\big) \in C^0([0,1],X) :g(0)=(0,0), g(1)=t_0V_0   \Big \}$. Clearly, we have $c>0$. For the case of $V_0=U_{\beta}$, we can derive that $$0<c<\frac{2s-\beta}{2(n-\beta)} {S(n,s,\beta)}^{\frac{n-\beta}{2s-\beta}}.$$
In fact, for $U_{\beta}=(u_{\beta},\tilde{u}_{\beta}) \in X$ and every $ t \geq 0$, we have
\begin{align*}
I(tU_{\beta}) \leq f_1(t):=\frac{t^2}{2}||(u_{\beta},\tilde{u}_{\beta})||^2-\frac{t^ {2^{*}_{s}(\beta)}}{ 2^{*}_{s}(\beta) }\int_{\R^{n}}     {\frac{{|u_{\beta}|}^{ {2^{*}_{s}}(\beta)}+{|\tilde{u}_{\beta}|}^{ {2^{*}_{s}}(\beta)}}{|x'|^{\beta}}}dx. 
\end{align*}
Straightforward computations yield that $f_1(t)$ attains its maximum at the point
 $$\tilde{t}=\Big(  \frac{||(u_{\beta},\tilde{u}_{\beta})||^2 }{  \int_{\R^{n}}     {\frac{{|u_{\beta}|}^{ {2^{*}_{s}}(\beta)}+{|\tilde{u}_{\beta}|}^{ {2^{*}_{s}}(\beta)}}{|x'|^{\beta}}} dx  }\Big)^{\frac{1}{ 2^{*}_{s}(\beta)-2 }}$$
and
\begin{align*}
\mathop {\sup }\limits_{t \geq 0} f_1(t)=\frac{2s-\beta}{2(n-\beta)} \Big[  \frac{ ||(u_{\beta},\tilde{u}_{\beta})||^2 }
   { ( \int_{\R^{n}}     {\frac{{|u_{\beta}|}^{ {2^{*}_{s}}(\beta)}+{|\tilde{u}_{\beta}|}^{ {2^{*}_{s}}(\beta)}}{|x'|^{\beta}}} dx )^{\frac{2}{ 2_s^*(\beta)}} }  \Big]^{\frac{n-\beta }{ 2s-\beta }}
=\frac{2s-\beta}{2(n-\beta)} {S(n,s,\beta)}^{\frac{n-\beta}{2s-\beta}}.
\end{align*}
We obtain that,
\begin{equation} \label{eq5.4}
 \mathop {\sup }\limits_{t \geq 0} I({tU_{\beta}}) \leq  \mathop {\sup }\limits_{t \geq 0} f_1(t)=\frac{2s-\beta}{2(n-\beta)} {S(n,s,\beta)}^{\frac{n-\beta}{2s-\beta}}.
\end{equation}
The equality does not hold in (\ref{eq5.4}), otherwise, we would have that $\mathop {\sup }\limits_{t \geq 0} I({tU_{\beta}}) =  \mathop {\sup }\limits_{t \geq 0} f_1(t)$. Let $t_1>0$ where $\mathop {\sup }\limits_{t \geq 0} I(tU_{\beta})$ is attained. We have
$$ f_1(t_1)-\frac{2 t_1^ {2^{*}_{s}(\alpha)}}{2^{*}_{s}(\alpha)}\int_{\R^n}{\frac{{|u_{\beta} \tilde{u}_{\beta}|}^{ \frac{{2^{*}_{s}}(\alpha)}{2} }}{|x'|^{\alpha}}}dx=f_1(\tilde{t})$$
which means that $f_1(t_1)>f_1(\tilde{t})$ since $t_1>0$. This contradicts the fact that $\tilde{t}$ is the unique maximum point of $f_1(t)$.
Thus
\begin{equation} \label{eq5.5}
 \mathop {\sup }\limits_{t \geq 0} I({tU_{\beta}})<  \mathop {\sup }\limits_{t \geq 0} f_1(t)=\frac{2s-\beta}{2(n-\beta)} {S(n,s,\beta)}^{\frac{n-\beta}{2s-\beta}}.
\end{equation}

For the case of $v_0=V_{\alpha}$, similarly, we can verify
\begin{equation} \label{eq5.6}
 \mathop {\sup }\limits_{t \geq 0} I({tV_{\alpha}}) <\frac{2s-\alpha}{n-\alpha} \Big[\frac{\Lambda(n,s,\alpha)}{2}\Big]^{\frac{n-\alpha}{2s-\alpha}}
\end{equation}
and thus $0<c<\frac{2s-\alpha}{n-\alpha} \Big[\frac{\Lambda(n,s,\alpha)}{2}\Big]^{\frac{n-\alpha}{2s-\alpha}}$.

From (\ref{eq5.5}) and (\ref{eq5.6}), we have
$0<c<c^*:=\min \Big \{ \frac{2s-\beta}{2(n-\beta)} {S(n,s,\beta)}^{\frac{n-\beta}{2s-\beta}}, \frac{2s-\alpha}{n-\alpha} \Big[\frac{\Lambda(n,s,\alpha)}{2}\Big]^{\frac{n-\alpha}{2s-\alpha}} \Big \}$.
Since (1)-(3) of Lemma \ref{lemma5.1} are satisfied, there exists a $(PS)$ sequence $\{(u_k,v_k)\} \subset X$ for $I$ at some $c\in (0,c^*)$, i.e.
\begin{equation*}
\lim_{k \to +\infty}I(u_k,v_k)=c~~\mbox{and}~~ \lim_{k \to +\infty} I'(u_k,v_k)=0~~\mbox{strongly in}~~X'.
\end{equation*}
\end{proof}

\noindent\textbf{Proof of Theorem \ref{th1.1}}

\noindent \textbf{(I)} The case $s \in(0,1)$, $0<\alpha,\beta<2s<n$, $2s<m<n$ and $\gamma_1,\gamma_2<\gamma_{H}$.

Let $\{(u_k,v_k)\}_{k\in \N}$ be a $(PS)$ sequence as in Proposition \ref{prop5.2}, i.e.
$$
I(u_k,v_k) \rightarrow c  ,~~ I'(u_k,v_k) \rightarrow 0 ~~\mbox{strongly in}~~X' ~~\mbox{as}~~ k \rightarrow +\infty.$$
Then, we have
\begin{equation}\label{eq5.7}
  I(u_k,v_k)\!=\!\frac{1}{2}||(u_k,v_k)||^2
\!-\!\frac{2}{2^{*}_{s}(\alpha)}\int_{\R^n}{\frac{{|u_kv_k|}^{ \frac{{2^{*}_{s}}(\alpha)}{2} }}{|x'|^{\alpha}}}dx
\!-\!\frac{1}{ 2^{*}_{s}(\beta) }\int_{\R^{n}}    {\frac{{|u_k|}^{ {2^{*}_{s}}(\beta)}\!+\!{|v_k|}^{ {2^{*}_{s}}(\beta)}}{|x'|^{\beta}}} dx\!\to\!c
\end{equation}
and
\begin{equation}\label{eq5.8}
  \langle I'(u_k,v_k),(u_k,v_k)\rangle_X\!=\!||(u_k,v_k)||^2
\!-\!2\int_{\R^n}{\frac{{|u_kv_k|}^{ \frac{{2^{*}_{s}}(\alpha)}{2} }}{|x'|^{\alpha}}}dx\!-\!\int_{\R^{n}}    {\frac{{|u_k|}^{ {2^{*}_{s}}(\beta)}\!+\!{|v_k|}^{ {2^{*}_{s}}(\beta)}}{|x'|^{\beta}}} dx\!\to\!0.
\end{equation}

From (\ref{eq5.7}) and (\ref{eq5.8}), if $\alpha\geq\beta$, we have
\begin{align*}
c+o(1)||(u_k,v_k)||&=I(u_k,v_k)-\frac{1}{2^{*}_{s}(\alpha)}  {\langle I'(u_k,v_k),(u_k,v_k) \rangle}_X
   \geq  \Big(\frac{1}{2}-\frac{1}{2^{*}_{s}(\alpha)} \Big)||(u_k,v_k)||^2.
\end{align*}
If $\alpha<\beta$, we have
\begin{align*}
c+o(1)||(u_k,v_k)||&=I(u_k,v_k)-\frac{1}{2^{*}_{s}(\beta)}  {\langle I'(u_k,v_k),(u_k,v_k) \rangle}_X  \geq  \Big(\frac{1}{2}-\frac{1}{2^{*}_{s}(\beta)} \Big)||(u_k,v_k)||^2.
\end{align*}
Thus, $\{(u_k,v_k)\}$ is bounded in $X$. From (\ref{eq5.8}), there exists a subsequence, still denoted by $\{(u_k,v_k)\}$, such that $||(u_k,v_k)||^2\!\rightarrow \!b$, $\int_{\R^n}{\frac{{|u_kv_k|}^{ \frac{{2^{*}_{s}}(\alpha)}{2} }}{|x'|^{\alpha}}}dx \!\rightarrow \!d_1 $, $\int_{\R^{n}}    {\frac{{|u_k|}^{ {2^{*}_{s}}(\beta)}\!+\!{|v_k|}^{ {2^{*}_{s}}(\beta)}}{|x'|^{\beta}}} dx\! \rightarrow \!d_2$ and $$b=2d_1+d_2.$$
By the definition of $\Lambda(n,s,\alpha)$ and $S(n,s,\beta)$, we get
\begin{equation*}
  d_1^{\frac{2}{  2^{*}_{s}(\alpha) }}\Lambda(n,s,\alpha) \leq b,\quad d_2^{\frac{2}{  2^{*}_{s}(\beta) }}S(n,s,\beta)\leq b,
\end{equation*}
which gives
$d_1^{\frac{2}{  2^{*}_{s}(\alpha) }}\Lambda(n,s,\alpha)  \leq 2d_1+d_2$ and $d_2^{\frac{2}{  2^{*}_{s}(\beta) }}S(n,s,\beta)\leq 2d_1+d_2$. As a result,
\begin{equation}\label{eq5.9}
  d_1^{\frac{2}{  2^{*}_{s}(\alpha) }}\Big(\Lambda(n,s,\alpha)-  2d_1^{\frac{2^{*}_{s}(\alpha)-2}{  2^{*}_{s}(\alpha) }}\Big) \leq d_2,~~~~d_2^{\frac{2}{  2^{*}_{s}(\beta) }}\Big(S(n,s,\beta)-  d_2^{\frac{2^{*}_{s}(\beta)-2}{  2^{*}_{s}(\beta) }} \Big)\leq 2d_1.
\end{equation}

We claim that $\Lambda(n,s,\alpha)-  2d_1^{\frac{2^{*}_{s}(\alpha)-2}{  2^{*}_{s}(\alpha) }}>0$ and $S(n,s,\beta)-  d_2^{\frac{2^{*}_{s}(\beta)-2}{  2^{*}_{s}(\beta) }}>0$. In fact, since $c+o(1)||(u_k,v_k)||=I(u_k,v_k)-\frac{1}{2}  {\langle I'(u_k,v_k),(u_k,v_k) \rangle}_X $, we have
\begin{align*}
\Big( 1-\frac{2}{2^{*}_{s}(\alpha)}\Big)
  \int_{\R^n}{\frac{{|u_kv_k|}^{ \frac{{2^{*}_{s}}(\alpha)}{2} }}{|x'|^{\alpha}}}dx+\Big(\frac{1}{ 2}-\frac{1}{ 2^{*}_{s}(\beta) }\Big)\int_{\R^{n}}  {\frac{{|u_k|}^{ {2^{*}_{s}}(\beta)}+{|v_k|}^{ {2^{*}_{s}}(\beta)}}{|x'|^{\beta}}} dx=c+o(1)||(u_k,v_k)||.
\end{align*}
Let $k\to+\infty$, we have
\begin{align} \label{eq5.010}
\Big( 1-\frac{2}{2^{*}_{s}(\alpha)}\Big)
 d_1+\Big(\frac{1}{ 2}-\frac{1}{ 2^{*}_{s}(\beta) }\Big) d_2=c,
\end{align}
which leads to $d_1  \leq  \frac{n-\alpha}{2s-\alpha} c$ and $d_2  \leq  \frac{2(n-\beta)}{2s-\beta} c$. Recall that $0<c<c^*$, we have
$$\Lambda(n,s,\alpha)-  2d_1^{\frac{2^{*}_{s}(\alpha)-2}{  2^{*}_{s}(\alpha) }}\geq A_1>0,~~~~S(n,s,\beta)-  d_2^{\frac{2^{*}_{s}(\beta)-2}{  2^{*}_{s}(\beta) }}\geq A_2>0,$$
where $A_1\!=\!\Lambda(n,s,\alpha)\!-\!2[\frac{n-\alpha}{2s-\alpha} c]^{\frac{2s-\alpha}{n-\alpha}}$ and $A_2\!=\!S(n,s,\beta)\!-\!  [\frac{2(n-\beta)}{2s-\beta} c]^{\frac{2s-\beta}{n-\beta}}$.
Thus (\ref{eq5.9}) imply
\begin{equation*}
  d_1^{\frac{2}{  2^{*}_{s}(\alpha) }}A_1\leq d_2,~~~~d_2^{\frac{2}{  2^{*}_{s}(\beta) }}A_2\leq 2d_1.
\end{equation*}

If $d_1=0$ and $d_2=0$, then (\ref{eq5.010}) implies that $c=0$, a contradiction with $c>0$. Therefore $d_1>0$ and $d_2>0$, we can choose $\varepsilon_0>0$ such that $d_1\geq\varepsilon_0>0$ and $d_2\geq\varepsilon_0>0$, so there exists a $K>0$ such that $k\geq K$ and
 $$\int_{\R^n}{\frac{{|(u_kv_k)(y)|}^{ \frac{{2^{*}_{s}}(\alpha)}{2} }}{|y'|^{\alpha}}}dy >\varepsilon_0/2,~~~~  \int_{\R^{n}}    {\frac{{|u_k(y)|}^{ {2^{*}_{s}}(\beta)}+{|v_k(y)|}^{ {2^{*}_{s}}(\beta)}}{|y'|^{\beta}}} dy>\varepsilon_0/2.$$
Then the embeddings (\ref{eq1.06}) and inequality \eqref{eqa1.6} imply that there exists $C>0$ such that
\begin{align*}
0<C \leq ||(u_kv_k)||_{  L^{1,n-2s+r}(\R^{n},|y'|^{-r}) }
\leq ||u_k||_{  L^{2,n-2s+r}(\R^{n},|y'|^{-r}) }
||v_k||_{  L^{2,n-2s+r}(\R^{n},|y'|^{-r}) }\leq C^{-1},
\end{align*}
where $r\!=\!\frac{2\alpha}{ 2^*_{s}(\alpha) }$. For any $k\!>\! K$, we may find ${\lambda}_k\!>\!0$ and $x_k\!=\!(x'_k,x''_k) \!\in\! \R^{m}\!\times\! \R^{n-m}$ such that
$$  {\lambda}_k^{-2s+r} \int_{B_{{\lambda}_k}(x_k)} \frac{|(u_kv_k)(y)|}{  |y'|^{r} }dy > ||(u_kv_k)||_{  L^{1,n-2s+r}(\R^{n},|y'|^{-r}) } -\frac{C}{2k} \geq C_1>0,$$
where $y=(y',y'')\in \R^{m}\times \R^{n-m}$.

Let $\tilde{u}_k(x)={\lambda}_k^{ \frac{n-2s}{2} }u_k\big({\lambda}_kx', {\lambda}_k(x''-x''_k)\big)$, $\tilde{v}_k(x)={\lambda}_k^{ \frac{n-2s}{2} }v_k\big({\lambda}_kx', {\lambda}_k(x''-x''_k)\big)$ and ${\tilde{x}}_k=\frac{x_k}{{\lambda}_k}
=({\tilde{x}}'_k,{\tilde{x}}''_k)\in \R^{m}\times \R^{n-m}$, then
\begin{equation}\label{eq4.4}
\int_{B_{1}(({\tilde{x}}'_k,\vec{0}))} \frac{|(\tilde{u}_k \tilde{v}_k)(x)|}{  |x'|^{r} }dx  \geq C_1>0,
\end{equation}
where $x=(x',x'')\in \R^{m}\times \R^{n-m}$ and $\vec{0}=(0,\cdots,0)\in \R^{n-m}$.

Since $||(\tilde{u}_k,\tilde{v}_k)||^2=||(u_k,v_k)||^2\leq C$, there exists $(\tilde{u},\tilde{v})\in X$ such that
$(\tilde{u}_k,\tilde{v}_k) \rightharpoonup (\tilde{u},\tilde{v}) ~~\mbox{in} ~~X=\dot{H}^{s}(\R^{n})\times\dot{H}^{s}(\R^{n})$. Similar to the proof of Proposition \ref{pro1.7}-(1), we can prove that $|{\tilde{x}}'_k|\leq C$, so there exists $R>0$ such that
$\int_{B_{R}(0)}  \frac{|(\tilde{u}_k \tilde{v}_k)(x)|}{  |x'|^{r} }dx  \geq C_1>0$. Recall that $\frac{r}{2}=\frac{\alpha}{ 2^*_{s}(\alpha) }<s$, from Lemma \ref{lemma2.3} we have $\int_{B_{R}(0)}  \frac{|(\tilde{u} \tilde{v})(x)|}{  |x'|^{r} }dx  \geq C_1>0$ and thus $\tilde{u}\not\equiv0$, $\tilde{v}\not\equiv0$.

In addition, the fact $||(\tilde{u}_k,\tilde{v}_k)||^2=||(u_k,v_k)||^2\leq C$ implies that $ \{ {|\tilde{u}_k|}^{\frac{{2^{*}_{s}}(\alpha)}{2} -2}\tilde{u}_k {|\tilde{v}_k|}^{\frac{{2^{*}_{s}}(\alpha)}{2} }  \}$ and $ \{ {|\tilde{v}_k|}^{\frac{{2^{*}_{s}}(\alpha)}{2} -2}\tilde{v}_k {|\tilde{u}_k|}^{\frac{{2^{*}_{s}}(\alpha)}{2} }  \}$ are bounded in $L^{     \frac{{2^{*}_{s}}({\alpha})}{{2^{*}_{s}({\alpha})}-1}    }(\R^{n},|x'|^{-{\alpha}})$, therefore,
\begin{align} \label{eq5.14}
 & {|\tilde{u}_k|}^{\frac{{2^{*}_{s}}(\alpha)}{2} -2}\tilde{u}_k {|\tilde{v}_k|}^{\frac{{2^{*}_{s}}(\alpha)}{2} } \rightharpoonup  {|\tilde{u}|}^{\frac{{2^{*}_{s}}(\alpha)}{2} -2}\tilde{u} {|\tilde{v}|}^{\frac{{2^{*}_{s}}(\alpha)}{2} }  ~~\mbox{in} ~~L^{     \frac{{2^{*}_{s}}({\alpha})}{{2^{*}_{s}({\alpha})}-1}    }(\R^{n},|x'|^{-{\alpha}}),
\end{align}
\begin{align} \label{eqa5.14}
 & {|\tilde{v}_k|}^{\frac{{2^{*}_{s}}(\alpha)}{2} -2}\tilde{v}_k {|\tilde{u}_k|}^{\frac{{2^{*}_{s}}(\alpha)}{2} } \rightharpoonup  {|\tilde{v}|}^{\frac{{2^{*}_{s}}(\alpha)}{2} -2}\tilde{v} {|\tilde{u}|}^{\frac{{2^{*}_{s}}(\alpha)}{2} }  ~~\mbox{in} ~~L^{     \frac{{2^{*}_{s}}({\alpha})}{{2^{*}_{s}({\alpha})}-1}    }(\R^{n},|x'|^{-{\alpha}}).
\end{align}
Similarly, we have
\begin{align} \label{eq5.15}
{|\tilde{u}_k|}^{{2^{*}_{s}}(\beta)-2}\tilde{u}_k \rightharpoonup  {|\tilde{u}|}^{{2^{*}_{s}}(\beta)-2}\tilde{u}  ~~\mbox{in} ~~L^{     \frac{{2^{*}_{s}}({\beta})}{{2^{*}_{s}({\beta})}-1}    }(\R^{n},|x'|^{-{\beta}}),
\end{align}
\begin{align} \label{eqa5.15}
{|\tilde{v}_k|}^{{2^{*}_{s}}(\beta)-2}\tilde{v}_k \rightharpoonup  {|\tilde{v}|}^{{2^{*}_{s}}(\beta)-2}\tilde{v}  ~~\mbox{in} ~~L^{     \frac{{2^{*}_{s}}({\beta})}{{2^{*}_{s}({\beta})}-1}    }(\R^{n},|x'|^{-{\beta}}).
\end{align}

Finally, we check that $\{(\tilde{u}_k,\tilde{v}_k)\}_{k\in \N}$ is also a $(PS)$ sequence for $I$ at level $c$. As the norms in ${\dot{H}}^s(\R^{n})$ and $L^{{2^{*}_{s}}(\alpha)}(\R^{n},|x'|^{-\alpha}) $ are invariant under the dilation and translation given by $\tilde{u}_k(x)={\lambda}_k^{ \frac{n-2s}{2} }u_k\big({\lambda}_kx', {\lambda}_k(x''-x''_k)\big)$ and $\tilde{v}_k(x)={\lambda}_k^{ \frac{n-2s}{2} }v_k\big({\lambda}_kx', {\lambda}_k(x''-x''_k)\big)$, we have
\begin{equation*}
 \lim_{k \to +\infty} I(\tilde{u}_k,\tilde{v}_k)=\lim_{k \to +\infty} I({u}_k,{v}_k)=c.
\end{equation*}
For any $(\phi,\psi)\! \in\! X$, we have $\big(\phi_k(x),\psi_k(x)\big)\!=\!\big(\lambda_k^{ \frac{2s-n}{2} }\phi(\frac{x'}{{\lambda}_k},\frac{x''}{{\lambda}_k}+x''_k),\lambda_k^{ \frac{2s-n}{2} }\psi(\frac{x'}{{\lambda}_k},\frac{x''}{{\lambda}_k}+x''_k)\big)\in X$, where $x=(x',x'')\in \R^{m}\times \R^{n-m}$. From $I'(u_k,v_k) \to 0 ~~\mbox{in}~~X'$, we can derive that
\begin{equation*}
 \lim_{k \to +\infty} {\langle I'(\tilde{u}_k,\tilde{v}_k),({\phi},{\psi}) \rangle}_X=\lim_{k \to +\infty} {\langle I'(u_k,v_k),({\phi}_k,{\psi}_k) \rangle}_X=0.
\end{equation*}
Thus, (\ref{eq5.14})-(\ref{eqa5.15}) lead to ${\langle I'(\tilde{u},\tilde{v}),({\phi},{\psi}) \rangle}_X\!=\!\lim_{k \to +\infty} {\langle I'(\tilde{u}_k,\tilde{v}_k),({\phi},{\psi}) \rangle}_X\!=\!0$. Since $\gamma_1\not=\gamma_2$, we have $\tilde{u}\not\equiv\tilde{v}$.
Hence $(\tilde{u},\tilde{v})$ is a nontrivial weak solution of (\ref{eq1.1}).

\noindent\textbf{(II)} The case $s \in(0,1)$, $\beta=0<\alpha<2s<n$, $2s<m<n$ and $0\leq \gamma_1,\gamma_2<\gamma_{H}$.

Since $\alpha>0$, (\ref{eq1.06}) and \eqref{eqa1.6} are still effective, so we can get a nontrivial weak solution to (\ref{eq1.1}) as above. Notice that $S(n,s,0)$ is attained provided $0\leq \gamma_1,\gamma_2<\gamma_{H}$.   \qed \\

\section{The extension to some general systems with partial weight}
In this Section, we extend the existence results of \eqref{eq1.1} to some general systems.

Problem I: let $\eta_1+\eta_2=2^*_{s}(\alpha)$ satisfy $1<\eta_1<\eta_2<\eta_1+\frac{\alpha}{s}$ and consider the existence of $
(u,v)\in X={\dot{H}}^s(\R^{n})\times{\dot{H}}^s(\R^{n}) $ to
\begin{equation} \label{Sye6.1}
\left\{ \begin{gathered}
  (-\Delta)^{s}u\!-\!{\gamma_1} {\frac{u}{|x'|^{2s}}}\!=\!{\frac{{|u|}^{ {2^{*}_{s}}(\beta)-2}u}{|x'|^{\beta}}}\!+\!\frac{\eta_1}{2^*_{s}(\alpha)}{\frac{{|u|}^{\eta_1 \!-\!2}u {|v|}^{\eta_2}}{|x'|^{\alpha}}}    \hfill \\
  (-\Delta)^{s}v\!-\!{\gamma_2} {\frac{v}{|x'|^{2s}}}\!=\!{\frac{{|v|}^{ {2^{*}_{s}}(\beta)-2}v}{|x'|^{\beta}}}\!+\!\frac{\eta_2}{2^*_{s}(\alpha)}{\frac{{|v|}^{ \eta_2\!-\!2}v {|u|}^{\eta_1 }}{|x'|^{\alpha}}}     \hfill \\
\end{gathered} \right.
\end{equation}
where $s \!\in\!(0,1)$, $0\!<\!\alpha,\beta\!<\!2s\!<\!n$, $0\!<\!m\!<\!n$, $x\!=\!(x',x'') \!\in\! \mathbb{R}^{m} \!\times\! \mathbb{R}^{n-m}$, $\gamma_1,\gamma_2\!<\!\gamma_{H}$, ${2^{*}_{s}}(\alpha)\!=\!\frac{2(n-\alpha)}{n-2s}$ and $\gamma_{H}$ was defined in Lemma \ref{lemma2.1}. 

Problem II: let $\eta_1+\eta_2=p^*(\alpha)$ satisfy $1<\eta_1<\eta_2<\eta_1+\alpha$ and consider the existence of $
(u,v)\in X={D}^{1,p}(\R^{n})\times{D}^{1,p}(\R^{n}) $ to
\begin{equation} \label{Sye6.2}
\left\{ \begin{gathered}
   -{\Delta}_pu\!-\!{\kappa_1} {\frac{|u|^{p-2}u}{|x'|^{p}}}\!=\!\frac{{|u|}^{p^{*}(\beta)-2}u}{|x'|^{\beta}}\!
   +\!\frac{\eta_1}{p^*(\alpha)}{\frac{{|u|}^{\eta_1 \!-\!2}u {|v|}^{\eta_2}}{|x'|^{\alpha}}}    \hfill \\
  -{\Delta}_pv\!-\!{\kappa_2} {\frac{|v|^{p-2}v}{|x'|^{p}}}\!=\!\frac{{|v|}^{p^{*}(\beta)-2}v}{|x'|^{\beta}}\!
  +\!\frac{\eta_2}{p^*(\alpha)}{\frac{{|v|}^{ \eta_2\!-\!2}v {|u|}^{\eta_1 }}{|x'|^{\alpha}}}     \hfill \\
\end{gathered} \right.
\end{equation}
where $n\geq 2$, $p \in (1,n)$, $0<\alpha,\beta<p<m<n$, $\kappa_1,\kappa_2< \tilde{\kappa}:=(\frac{m-p}{p})^p$, $x\!=\!(x',x'') \in \mathbb{R}^{m} \times \mathbb{R}^{n-m}$, $ p^{*}(\alpha)=\frac{p(n-\alpha)}{n-p}$. Here $\tilde{\kappa}$ was defined in formula (12) of \cite{mBgT}. \\


Then, the following main results hold:
\begin{theorem}\label{th6.1}
Let $\eta_1+\eta_2=2^*_{s}(\alpha)$ satisfy $1<\eta_1<\eta_2<\eta_1+\frac{\alpha}{s}$. Then system (\ref{Sye6.1}) possesses at least a nontrivial weak solution provided either \textbf{(I)} $s \in(0,1)$, $0<\alpha,\beta<2s<m<n$ and $\gamma_1,\gamma_2<\gamma_{H}$ or \textbf{(II)} $s \in(0,1)$, $\beta=0<\alpha<2s<m<n$ and $0\leq \gamma_1,\gamma_2<\gamma_{H}$.\\
\end{theorem}

\begin{theorem}\label{th6.2}
Let $\eta_1+\eta_2=p^*(\alpha)$ satisfy $1<\eta_1<\eta_2<\eta_1+\alpha$. Then system (\ref{Sye6.2}) possesses at least a nontrivial weak solution $(u,v)$ provided either \textbf{(I)} $n\geq 2$, $p \in [2,n)$, $0<\alpha,\beta<p<m<n$ and $\kappa_1,\kappa_2 < \tilde{\kappa}$ or \textbf{(II)} $n\geq 2$, $p \in [2,n)$, $\beta=0<\alpha<p<m<n$ and $0\leq\kappa_1,\kappa_2< \tilde{\kappa}$. Moreover, we have $u,v\in  {D}^{1,p}(\R^{n})\cap C^{1}(\mathbb{R}^{n} \setminus\{0\})$.\\
\end{theorem}

Before proving Theorems \ref{th6.1}-\ref{th6.2}, we set up the following improved Sobolev inequalities:
\begin{proposition}\label{prop6.1}
Let $s \!\in\! (0,1)$, $0\!<\!\alpha\!<\!2s\!<\!n$, $2s\!<\!m\!<\!n$, $\eta_1\!+\!\eta_2\!=\!2^*_{s}(\alpha)$ satisfy $1\!<\!\eta_1\!<\!\eta_2\!<\!\eta_1+\frac{\alpha}{s}$. There exists $C\!=\!C(n,m,s,\alpha,\eta_1,\eta_2)\!>\!0$ such that for any $u,v \!\in\! {\dot{H}}^s(\R^{n})$ and for any $\theta \!\in\! (\bar{\theta},\frac{2\eta_1}{2^*_{s}(\alpha)})$, it holds that
\begin{equation*}  
 \Big( \int_{ \R^{n} }  \frac{ |u(y)|^{\eta_1} |v(y)|^{\eta_2} }  {  |y'|^{\alpha} } dy  \Big)^{ \frac{1}{  2^*_{s} (\alpha)  }}  \leq C ||u||_{{\dot{H}}^s(\R^{n})}^{\frac{\theta}{2}}
 ||v||_{{\dot{H}}^s(\R^{n})}^{\frac{\theta}{2}+\frac{\eta_2-\eta_1}{2^*_{s} (\alpha)}} ||(uv)||^{\frac{\eta_1}{2^*_{s} (\alpha)}-\frac{\theta}{2}}_{  L^{1,n-2s+r}(\R^{n},|y'|^{-r}) },
\end{equation*}
where $y\!=\!(y',y'') \in \mathbb{R}^{m} \times \mathbb{R}^{n-m}$,  $\bar{\theta}=\max \Big\{ \frac{2}{2^*_{s}(\alpha)},\frac{2\eta_1-\frac{t\alpha}{s}}{2^*_{s}(\alpha)}, \frac{2\eta_1}{2^*_{s}(\alpha)}
-\frac{2t(\frac{\alpha}{2s}-\frac{\alpha}{m})}{2^*_{s}(\alpha)
-\frac{2\alpha}{m}},
\frac{2\eta_1}{2^*_{s}(\alpha)}-t \Big\}$, $t=1-\frac{(\eta_2-\eta_1)s}{\alpha}$ and $r=\frac{2\alpha}{ 2^*_{s}(\alpha) }$.
\end{proposition}


\begin{corollary}\label{coRo6.2}
Let $n\!\geq\!2$, $2\!\leq\! p\!<\!n$, $0\!<\!\alpha\!<\!p\!<\!m\!<\!n$, $\eta_1\!+\!\eta_2\!=\!p^*(\alpha)\!=\!\frac{p(n-\alpha)}{n-p}$ satisfy $1\!<\!\eta_1\!<\!\eta_2\!<\!\eta_1\!+\!\alpha$. There exists $C\!=\!C(n,m,p,\alpha,\eta_1,\eta_2)\!>\!0$ such that for any $u,v \in {D}^{1,p}(\R^{n})$ and for any $\theta \!\in\! (\bar{\theta},\frac{2\eta_1}{p^*(\alpha)})$, it holds that
\begin{equation*}  
  \Big( \int_{ \R^{n} }  \frac{ |u(y)|^{\eta_1} |v(y)|^{\eta_2} }  {  |y'|^{\alpha} } dy  \Big)^{ \frac{1}{  p^* (\alpha)  }}  \leq C ||u||_{{D}^{1,p}(\R^{n})}^{\frac{\theta}{2}}
 ||v||_{{D}^{1,p}(\R^{n})}^{\frac{\theta}{2}+\frac{\eta_2-\eta_1}{p^* (\alpha)}} ||(uv)||^{\frac{\eta_1}{p^*(\alpha)}-\frac{\theta}{2}}_{  L^{{p}/{2},n-p+r}(\R^{n},|y'|^{-r}) },
\end{equation*}
where $y\!=\!(y',y'') \in \mathbb{R}^{m} \times \mathbb{R}^{n-m}$, $\bar{\theta}=\max \Big\{ \frac{p}{p^*(\alpha)}, \frac{2\eta_1-t\alpha}{p^*(\alpha)}, \frac{2\eta_1}{p^*(\alpha)}-\frac{t\alpha(1-\frac{p}{n})}{p^*(\alpha)
-\frac{p\alpha}{n}},
\frac{2\eta_1}{p^*(\alpha)}-t \Big\}$, $t=1-\frac{\eta_2-\eta_1}{\alpha}$ and $r=\frac{p\alpha}{p^{*}(\alpha)}$.  \\
\end{corollary}

\noindent\textbf{Proof of  Proposition \ref{prop6.1}}

For $u \in {\dot{H}}^s(\R^{n})$, we have $\hat{g}(\xi):=\!|\xi|^s\hat{u}(\xi) \!\in\! L^2(\R^{n})$ and
$||u||_{{\dot{H}}^s(\R^{n})}\!=\!||g||_{L^2(\R^{n})}$  by Plancherel's theorem. Thus, $u(x)\!=\!(\frac{1}{|\xi|^s})^{\vee}*g(x)\!=\!{\ell}_sg(x)$, where ${\ell}_sg(x)\!=\!\int_{{\R}^n} \frac{g(z)}{|x-z|^{n-{s}}}dz$. Similarly, let $v \!\in\! {\dot{H}}^s(\R^{n})$ and $\hat{f}(\xi):=\!|\xi|^s\hat{v}(\xi) \!\in\! L^2(\R^{n})$, we have $v(x)\!=\!(\frac{1}{|\xi|^s})^{\vee}*f(x)\!=\!{\ell}_sf(x)$.

Firstly, let $2s<m<n$, $\eta_1+\eta_2=2^*_{s}(\alpha)$ with $1<\eta_1<\eta_2<\eta_1+\frac{\alpha}{s}$ and define
$$t:=1-\frac{(\eta_2-\eta_1)s}{\alpha}\in (0,1).$$
Take $\tilde{s}\!=\!s$, $\tilde{p}\!=\!2$, $\sigma\!=\!\frac{2s}{\alpha}\!>\!1$,
$\max \Big\{ 2, 2\eta_1-\frac{t\alpha}{s}, 2\eta_1-\frac{2t(\frac{1}{\sigma}-\frac{\alpha}{m})}{2^*_{s}(\alpha)
-\frac{2\alpha}{m}}\cdot 2^*_{s}(\alpha),
2\eta_1-t\cdot 2^*_{s}(\alpha) \Big\} \!<\! {\tilde{q}}\!<\! 2\eta_1$,
$W(y)\equiv 1$ and $V(y)=\frac{  {|u(y)|}^  {\eta_1-\frac{\tilde{q}}{2} } {|v(y)|}^  {\eta_2-\frac{\tilde{q}}{2} }  }   {  |y'|^{\alpha}  }$ in Lemma \ref{lemma3.4}, where $y\!=\!(y',y'') \in \mathbb{R}^{m} \times \mathbb{R}^{n-m}$. Then (\ref{eq3.3}) becomes
\begin{equation}\label{Qa6.3}
|Q|^{ \frac{s}{n}+\frac{1}{\tilde{q}}-\frac{1}{2}} { \Big(  \frac{1}{|Q|} \int_{Q}V^{\sigma}dy    \Big) }^{ \frac{1}{{\tilde{q}} \sigma} }  \leq C_{\sigma}.
\end{equation}

Secondly, we verify condition (\ref{eq3.3}). Since $2\eta_1-t\cdot 2^*_{s}(\alpha)< {\tilde{q}}< 2\eta_1$, we can define
$$
 \rho:=1-\frac{2\eta_1-\tilde{q}}{2^{*}_{s}(\alpha)}\cdot \frac{1}{t} \in (0,1).$$
For any fixed $x\in {\R}^n$, replacing $Q$ by ball $B_{R}(x)$, we deduce by H\"{o}lder's inequality that
\begin{align}\label{Qa6.4}
R^{-n} \int_{B_R(x)}V^{\sigma}dy&=R^{-n} \int_{B_R(x)}\frac{  {|u|}^  {(\eta_1-\frac{\tilde{q}}{2})\sigma } {|v|}^  {(\eta_2-\frac{\tilde{q}}{2})\sigma }  }   {  |y'|^{\sigma \alpha}  }dy=R^{-n} \int_{B_R(x)} \frac{{|uv|}^  {(\eta_1-\frac{\tilde{q}}{2})\sigma } }{{|y'|^  { t \sigma {\alpha}    }   }  } \cdot {\frac{{|v|}^{  (\eta_2-\eta_1)\sigma }}{|y'|^  { (1-t) \sigma {\alpha}    }   }}dy \nonumber \\
&\leq R^{-n} \Big[\int_{B_R(x)} \frac{{|uv|}^  {(\eta_1-\frac{\tilde{q}}{2})\frac{\sigma}{t} } }{{|y'|^  {  \sigma {\alpha}    }   }  } dy\Big]^{t} \Big[\int_{B_R(x)} {\frac{{|v|}^{  \frac{(\eta_2-\eta_1)\sigma}{1-t} }}{|y'|^  {  \sigma {\alpha}    }   }}dy\Big]^{1-t} \nonumber \\
&= R^{-n} \Big[\int_{B_R(x)} \frac{{|uv|}^  {(\eta_1-\frac{\tilde{q}}{2})\frac{\sigma}{t} } }{{|y'|^  {  \sigma {\alpha}    }   }  } dy\Big]^{t} \Big[\int_{B_R(x)} {\frac{{|v|}^{  2 }}{|y'|^  {  2s   }   }}dy\Big]^{1-t}  \nonumber\\
&\leq R^{-n} \Big[\int_{B_R(x)} \frac{{|uv|}^  {(\eta_1-\frac{\tilde{q}}{2})\frac{\sigma}{t} } }{{|y'|^  {  \sigma {\alpha}    }   }  } dy\Big]^{t} \Big[C||v||_{{\dot{H}}^s(\R^{n})}^{2}\Big]^{1-t} .
\end{align}
From $2\eta_1-\frac{t\alpha}{s} < {\tilde{q}}< 2\eta_1$, we have $0<(\eta_1-\frac{\tilde{q}}{2}){\frac{\sigma}{t} }
<1$. Moreover, we have $\frac{  \sigma {\alpha}\rho }{ 1- (\eta_1-\frac{\tilde{q}}{2}){\frac{\sigma}{t} }
 }<m$ by ${\tilde{q}}>2\eta_1-\frac{2t(\frac{1}{\sigma}-\frac{\alpha}{m})}{2^*_{s}(\alpha)
-\frac{2\alpha}{m}}\cdot 2^*_{s}(\alpha)$. Recall that $\rho\in (0,1)$, it results to 
\begin{align} \label{Qa6.5}
&\int_{B_R(x)} \frac{{|uv|}^  {(\eta_1-\frac{\tilde{q}}{2})\frac{\sigma}{t} } }{{|y'|^  {   \sigma {\alpha}    }   }  } dy=\int_{B_R(x)} \frac{1 }{{|y'|^  {  \rho\sigma {\alpha}    }   }  } \frac{{|uv|}^  {(\eta_1-\frac{\tilde{q}}{2})\frac{\sigma}{t} } }{{|y'|^  { (1-\rho) \sigma {\alpha}    }   }  } dy  \nonumber \\
& \leq \Big[\int_{B_R(x)} \frac{1 }{{|y'|^  {  \frac{\rho\sigma {\alpha}}{1-(\eta_1-\frac{\tilde{q}}{2}){\frac{\sigma}{t} }}    }   }  } dy\Big]^{1-(\eta_1-\frac{\tilde{q}}{2}){\frac{\sigma}{t} }}  \Big[\int_{B_R(x)} \frac{{|uv|}}{{|y'|^  { \frac{t{\alpha}(1-\rho) }{(\eta_1-\frac{\tilde{q}}{2})}    }   }  } dy\Big]^{(\eta_1-\frac{\tilde{q}}{2}){\frac{\sigma}{t} }} \nonumber\\
& = CR^{-\rho \sigma \alpha +n -n(\eta_1-\frac{\tilde{q}}{2}){\frac{\sigma}{t} }} \Big[\int_{B_R(x)} \frac{{|uv|}}{{|y'|^  r   }  } dy\Big]^{(\eta_1-\frac{\tilde{q}}{2}){\frac{\sigma}{t} }},
\end{align}
where $r:=\frac{t{\alpha}(1-\rho) }{\eta_1-\frac{\tilde{q}}{2}} =\frac{2\alpha}{2^{*}_{s}(\alpha)}$. Put \eqref{Qa6.5} into \eqref{Qa6.4}, we get
\begin{align*}
R^{-n} \int_{B_R(x)}V^{\sigma}dy&\leq C R^{-\rho \sigma \alpha t +n(t-1) -n(\eta_1-\frac{\tilde{q}}{2})\sigma}   \Big[\int_{B_R(x)} \frac{{|uv|}}{{|y'|^  r   }  } dy\Big]^{(\eta_1-\frac{\tilde{q}}{2})\sigma} ||v||_{{\dot{H}}^s(\R^{n})}^{2(1-t)}.
\end{align*}
Therefore,
\begin{align*}
&R^{s+\frac{n}{{\tilde{q}}}-\frac{n}{2}} { \Big(  R^{-n} \int_{ B_R(x) }V^{\sigma}dy    \Big) }^{ \frac{1}{{\tilde{q}} \sigma} }  \\
\leq & R^{s+\frac{n}{{\tilde{q}}}-\frac{n}{2}} \Big\{ C R^{-\rho \sigma \alpha t +n(t-1) -n(\eta_1-\frac{\tilde{q}}{2})\sigma}   \Big[\int_{B_R(x)} \frac{{|uv|}}{{|y'|^  r   }  } dy\Big]^{(\eta_1-\frac{\tilde{q}}{2})\sigma} ||v||_{{\dot{H}}^s(\R^{n})}^{2(1-t)}  \Big\}^{\frac{1}{{\tilde{q}} \sigma}}  \\
= & C \Big\{ R^{ \tilde{q} \sigma(s+\frac{n}{{\tilde{q}}}-\frac{n}{2})-\rho \sigma \alpha t +n(t-1) -n(\eta_1-\frac{\tilde{q}}{2})\sigma}   \Big[\int_{B_R(x)} \frac{{|uv|}}{{|y'|^  r   }  } dy\Big]^{(\eta_1-\frac{\tilde{q}}{2})\sigma} ||v||_{{\dot{H}}^s(\R^{n})}^{2(1-t)}  \Big\}^{\frac{1}{{\tilde{q}} \sigma}}  \\
= & C \Big\{ R^{ \tilde{q} \sigma(s+\frac{n}{{\tilde{q}}}-\frac{n}{2})-\rho \sigma \alpha t +n(t-1) -n(\eta_1-\frac{\tilde{q}}{2})\sigma}   \Big[\int_{B_R(x)} \frac{{|uv|}}{{|y'|^  r   }  } dy\Big]^{(\eta_1-\frac{\tilde{q}}{2})\sigma} \Big\}^{\frac{1}{{\tilde{q}} \sigma}} ||v||_{{\dot{H}}^s(\R^{n})}^{\frac{{2^{*}_{s}(\alpha)}-2\eta_1}{\tilde{q}}} \\
= & C \Big\{  R^{  \frac{\tilde{q} \sigma(s+\frac{n}{{\tilde{q}}}-\frac{n}{2})-\rho \sigma \alpha t +n(t-1) -n(\eta_1-\frac{\tilde{q}}{2})\sigma}{{(\eta_1-\frac{\tilde{q}}{2})\sigma}}   }  \cdot  {  \int_{B_R(x)}  {\frac{{|uv|}}{|y'|^r }}dy   }   \Big\}^{  \frac{(\eta_1-\frac{\tilde{q}}{2})}{\tilde{q}}  }  ||v||_{{\dot{H}}^s(\R^{n})}^{\frac{{2^{*}_{s}(\alpha)}-2\eta_1}{\tilde{q}}} \\
= & C \Big\{  R^{  \frac{\tilde{q} \sigma(s+\frac{n}{{\tilde{q}}}-\frac{n}{2})-\rho \sigma \alpha t +n(t-1) }{{(\eta_1-\frac{\tilde{q}}{2})\sigma}}   }  \cdot R^{-n}\cdot  {  \int_{B_R(x)}  {\frac{{|uv|}}{|y'|^r }}dy   }   \Big\}^{  \frac{(\eta_1-\frac{\tilde{q}}{2})}{\tilde{q}}  }  ||v||_{{\dot{H}}^s(\R^{n})}^{\frac{{2^{*}_{s}(\alpha)}-2\eta_1}{\tilde{q}}} \\
= & C \Big\{  R^{   n-2s+r }  \cdot R^{-n}  \cdot {  \int_{B_R(x)}  {\frac{{|uv|}}{|y'|^  { r }   }}dy   }   \Big\}^{  \frac{(\eta_1-\frac{\tilde{q}}{2})}{\tilde{q}}  } ||v||_{{\dot{H}}^s(\R^{n})}^{\frac{{2^{*}_{s}(\alpha)}-2\eta_1}{\tilde{q}}}\\
=&  C {||(uv)||}^{  \frac{(\eta_1-\frac{\tilde{q}}{2})}{\tilde{q}}  }_{ L^{1,n-2s+r }(\R^{n},|y'|^{-r})}
||v||_{{\dot{H}}^s(\R^{n})}^{\frac{{2^{*}_{s}(\alpha)}-2\eta_1}{\tilde{q}}}
:=C_{\sigma}.
\end{align*}
Since $u={\ell}_sg$ and $v={\ell}_sf$, and by Lemma \ref{lemma3.4},
\begin{align*}
&\int_{ \R^{n} }  \frac{ |u(y)|^{\eta_1} |v(y)|^{\eta_2} }  {  |y'|^{\alpha} } dy = \int_{ \R^{n} } V(y) {|{\ell}_sg(y)|}^{\frac{\tilde{q}}{2}}{|{\ell}_sf(y)|}^{\frac{\tilde{q}}{2}} dy \\
&\leq \Big[\int_{ \R^{n} } V(y) {|{\ell}_sg(y)|}^{\tilde{q}} dy\Big]^{\frac{1}{2}} \Big[\int_{ \R^{n} } V(y) {|{\ell}_sf(y)|}^{\tilde{q}} dy\Big]^{\frac{1}{2}}   \\
& \leq (CC_{\sigma})^{\tilde{q}} ||g||^{\frac{\tilde{q}}{2}}_{L^2} ||f||^{\frac{\tilde{q}}{2}}_{L^2}   \leq C ||u||_{{\dot{H}}^s(\R^{n})}^{\frac{\tilde{q}}{2}}
 ||v||_{{\dot{H}}^s(\R^{n})}^{\frac{\tilde{q}}{2}} {||(uv)||}^{\eta_1-\frac{\tilde{q}}{2}}_{ L^{1,n-2s+r }(\R^{n},|y'|^{-r})}
||v||_{{\dot{H}}^s(\R^{n})}^{{2^{*}_{s}(\alpha)}-2\eta_1}\\
& = C ||u||_{{\dot{H}}^s(\R^{n})}^{\frac{\tilde{q}}{2}}
 ||v||_{{\dot{H}}^s(\R^{n})}^{\frac{\tilde{q}}{2}+(\eta_2-\eta_1)} {||(uv)||}^{\eta_1-\frac{\tilde{q}}{2}}_{ L^{1,n-2s+r }(\R^{n},|y'|^{-r})}.
\end{align*}
Then, for any $\theta=\frac{{\tilde{q}}}{ 2^*_{s}(\alpha) }$ satisfying $\bar{\theta}<\theta<\frac{2\eta_1}{2^*_{s}(\alpha)}<1$, we have
\begin{equation*}
 \Big( \int_{ \R^{n} }  \frac{ |u(y)|^{\eta_1} |v(y)|^{\eta_2} }  {  |y'|^{\alpha} } dy  \Big)^{ \frac{1}{  2^*_{s} (\alpha)  }}  \leq C ||u||_{{\dot{H}}^s(\R^{n})}^{\frac{\theta}{2}}
 ||v||_{{\dot{H}}^s(\R^{n})}^{\frac{\theta}{2}+\frac{\eta_2-\eta_1}{2^*_{s} (\alpha)}} ||(uv)||^{\frac{\eta_1}{2^*_{s} (\alpha)}-\frac{\theta}{2}}_{  L^{1,n-2s+r}(\R^{n},|y'|^{-r}) },
\end{equation*}
where $s \in (0,1)$, $0<\alpha<2s<m<n$, $r=\frac{2\alpha}{ 2^*_{s}(\alpha) }$, $C=C(n,m,s,\alpha,\eta_1,\eta_2)>0$ is a constant, $\bar{\theta}=\max \Big\{ \frac{2}{2^*_{s}(\alpha)},\frac{2\eta_1-\frac{t\alpha}{s}}{2^*_{s}(\alpha)}, \frac{2\eta_1}{2^*_{s}(\alpha)}
-\frac{2t(\frac{\alpha}{2s}-\frac{\alpha}{m})}{2^*_{s}(\alpha)
-\frac{2\alpha}{m}},
\frac{2\eta_1}{2^*_{s}(\alpha)}-t \Big\}$ and $t=1-\frac{(\eta_2-\eta_1)s}{\alpha}$.
\qed  \\



\noindent\textbf{Proof of Corollary \ref{coRo6.2}}

For any $u \in {D}^{1,p}(\R^{n})(n\geq 2)$, recall that $|u(x)| \!\leq\! C {\ell}_{1}(| {\nabla u} |)(x)$ for some $C=C(n)>0$.

Let $2\leq p<m<n$, $\eta_1+\eta_2=p^*(\alpha)$ with $1<\eta_1<\eta_2<\eta_1+\alpha$ and denote
$$t:=1-\frac{\eta_2-\eta_1}{\alpha}\in (0,1).$$
Take $\tilde{s}\!=\!1$, $\tilde{p}\!=\!p$, $\sigma\!=\!\frac{p}{\alpha}\!>\!1$, $\max \Big\{ p, 2\eta_1-t\alpha, 2\eta_1-\frac{t\alpha(1-\frac{p}{m})}{p^*(\alpha)
-\frac{p\alpha}{m}}\cdot p^*(\alpha),
2\eta_1-t\cdot p^*(\alpha) \Big\} \!<\! {\tilde{q}}\!<\! 2\eta_1$,
$W(y)\equiv 1$ and $V(y)=\frac{  {|u(y)|}^  {\eta_1-\frac{\tilde{q}}{2} } {|v(y)|}^  {\eta_2-\frac{\tilde{q}}{2} }  }   {  |y'|^{\alpha}  }$ in Lemma \ref{lemma3.4}, where $y\!=\!(y',y'') \in \mathbb{R}^{m} \times \mathbb{R}^{n-m}$. The remain argument is similar to the case in $ {\dot{H}}^s(\R^{n})$ with $\rho:=1-\frac{2\eta_1-\tilde{q}}{p^{*}(\alpha)}\cdot \frac{1}{t} \in (0,1)$ while $r:=\frac{tp{\alpha}(1-\rho) }{2\eta_1-\tilde{q}} =\frac{p\alpha}{p^{*}(\alpha)}$. Actually, for $\theta=\frac{{\tilde{q}}}{ p^*(\alpha) }$ satisfying $\bar{\theta}<\theta<\frac{2\eta_1}{p^*(\alpha)}$, we have
\begin{equation*}
 \Big( \int_{ \R^{n} }  \frac{ |u|^{\eta_1} |v|^{\eta_2} }  {  |y'|^{\alpha} } dy  \Big)^{ \frac{1}{  p^* (\alpha)  }}  \leq C ||u||_{{D}^{1,p}(\R^{n})}^{\frac{\theta}{2}}
 ||v||_{{D}^{1,p}(\R^{n})}^{\frac{\theta}{2}+\frac{\eta_2-\eta_1}{p^* (\alpha)}} ||(uv)||^{\frac{\eta_1}{p^*(\alpha)}-\frac{\theta}{2}}_{  L^{{p}/{2},n-p+r}(\R^{n},|y'|^{-r}) },
\end{equation*}
where $r=\frac{p\alpha}{p^{*}(\alpha)}$, $\bar{\theta}=\max \Big\{ \frac{p}{p^*(\alpha)}, \frac{2\eta_1-t\alpha}{p^*(\alpha)}, \frac{2\eta_1}{p^*(\alpha)}-\frac{t\alpha(1-\frac{p}{m})}{p^*(\alpha)
-\frac{p\alpha}{m}},
\frac{2\eta_1}{p^*(\alpha)}-t \Big\}$ and $t=1-\frac{\eta_2-\eta_1}{\alpha}$.  \qed  \\

\noindent\textbf{Proof of  Theorems \ref{th6.1}-\ref{th6.2}}

Imitate the proof of Theorem \ref{th1.1}, we shall replace Proposition \ref{prop1.4} by Proposition \ref{prop6.1} or Corollary \ref{coRo6.2} respectively, in proving Theorems \ref{th6.1}-\ref{th6.2}. It is standard as in \cite{RFPP} that $u,v\in  {D}^{1,p}(\R^{n})\cap C^{1}(\mathbb{R}^{n} \setminus\{0\})$ in (\ref{Sye6.2}).\qed  \\

\end{document}